\documentclass[11pt,english]{smfart}
	
\usepackage{amsmath,leftidx,amsthm,sfheaders}
\usepackage{xspace}
\usepackage[all]{xy}
\usepackage[psamsfonts]{amssymb}
\usepackage[latin1]{inputenc}
\usepackage{graphicx,color,mathrsfs}
\usepackage{hyperref}
\usepackage[french]{babel}

\newtheorem*{remark}{\bf Remark}

\theoremstyle{plain}

\newtheorem{theorem}{\bf Theorem}[section]
\newtheorem{proposition}[theorem]{\bf Proposition}
\newtheorem{definition}[theorem]{\bf Definition}
\newtheorem{Theorem}{\bf Theorem}

\newtheorem*{claim}{\bf Claim}
\newtheorem{lemma}[theorem]{\bf Lemma}

\newtheorem{question}{\bf Question}

\def\C{{\mathbb C}}

\def\R{{\mathbb R}}

\def\B{\mathbb{B}}
\def\D{\mathbb{D}}

\def\p{\mathbb{P}}

\def\supp{\textup{supp}}

\def\un{\underline{n}}

\title[Dynamical pairs]{Dynamical pairs with an absolutely continuous bifurcation measure}
\author{Thomas Gauthier}
\address{CMLS, \'Ecole Polytechnique, Institut Polytechnique de Paris, 91128 Palaiseau Cedex, France}
\email{thomas.gauthier@polytechnique.edu}

\begin{document}

\begin{abstract}
In this article, we study algebraic dynamical pairs $(f,a)$  parametrized by an irreducible quasi-projective curve $\Lambda$ having an absolutely continuous bifurcation measure. We prove that, if $f$ is non-isotrivial and $(f,a)$ is unstable, this is equivalent to the fact that $f$ is a family of Latt\`es maps.
To do so, we prove the density of transversely prerepelling parameters in the bifurcation locus of $(f,a)$ and a similarity property, at any transversely prerepelling parameter $\lambda_0$, between the measure $\mu_{f,a}$ and the maximal entropy measure of $f_{\lambda_0}$.
We also establish an equivalent result for dynamical pairs of $\p^k$, under an additional mild assumption.
\end{abstract}

\begin{altabstract}
Dans cet article, nous \'etudions les paires dynamiques $(f,a)$ alg\'ebriques param\'etr\'ees par une courbe quasi-projective irr\'eductible poss\'edant une mesure de bifurcation absolument continue. Nous prouvons que, si la famille $f$ n'est pas isotriviale et si la paire $(f,a)$ est instable, c'est \'equivalent au fait que la famille $f$ soit une famille d'exemples de Latt\`es flexibles. A cette fin, nous montrons la densit\'e des param\`etres transversalement pr\'er\'epulsifs dans le lieu de bifurcation de la paire $(f,a)$, ainsi qu'une propri\'et\'e de similarit\'e, en un param\`etre transversalement pr\'er\'epulsif $\lambda_0$, entre la mesure de bifurcation $\mu_{f,a}$ et la mesure d'entropie maximale de $f_{\lambda_0}$. Sous une hypoth\`ese relativement g\'en\'erale, nous \'etablissons \'egalement un r\'esultat similaire pour les paires dynamiques de $\mathbb{P}^k$.
\end{altabstract}

\maketitle

\section*{Introduction}

Let $\Lambda$ be a complex manifold. A \emph{dynamical pair} $(f,a)$ parametrized by $\Lambda$ is  a holomorphic family $f : \Lambda \times\p^{1}\to\p^{1}$ of rational maps of degree $d\geq2$, i.e. $f$ is a holomorphic and $f_\lambda$ is a degree $d$ rational map for all $\lambda\in\Lambda$,
together with a marked point $a$, i.e. a holomorphic map $a:\Lambda\to\p^{1}$.

\smallskip

Recall that a dynamical pair $(f,a)$ of the Riemann sphere is \emph{stable} if the sequence $\{\lambda\mapsto f_\lambda^n(a(\lambda))\}_{n\geq1}$ is a normal family on $\Lambda$. Otherwise, we say  that the pair $(f,a)$ is \emph{unstable}. Recall also that $f$ is \emph{isotrivial} if there exists a branched cover $X\to\Lambda$ and a holomorphic family of M\"obius transformations $M:X\times\p^1\to\p^1$ so that $M_\lambda\circ f_\lambda\circ M_\lambda^{-1}:\p^1\to\p^1$ is independent of the parameter $\lambda$ and that the pair  $(f,a)$ is \emph{isotrivial} if, in addition, $M_\lambda(a(\lambda))$ is also independent of the parameter $\lambda$.

~

When a dynamical pair $(f,a)$ is unstable, the \emph{stability locus} $\mathrm{Stab}(f,a)$ is the set of points $\lambda_0\in\Lambda$ admitting a neighborhood $U$ on which the pair $(f,a)$ the sequence $\{\lambda\mapsto f_\lambda^n(a(\lambda))\}_{n\geq1}$ is a normal family.
The \emph{bifurcation locus} $\mathrm{Bif}(f,a)$ of the pair $(f,a)$ is its complement $\mathrm{Bif}(f,a):=\Lambda\setminus\mathrm{Stab}(f,a)$. 
If $a$ is the marking of a critical point, i.e. $f_\lambda'(a(\lambda))=0$ for all $\lambda\in\Lambda$, it is classical that the bifurcation locus $\mathrm{Bif}(f,a)$ has empty interior, \cite{MSS}. However, when $f$ is not a family of polynomials and $a$ is not a marked critical point, $\mathrm{Bif}(f,a)$ can have non-empty interior. For instance, if $f$ is an isotrivial family with $f_\lambda=f_{\lambda'}$ for all $\lambda,\lambda'\in\Lambda$ and $J_f=\p^1$, then $\mathrm{Bif}(f,a)$ is either empty or the whole parameter space $\Lambda$. In fact, we can describe precisely when $\mathrm{Bif}(f,a)$ can have non-empty interior.

We say that a family $f:\Lambda\times\p^1\longrightarrow \p^1$ of degree $d$ rational maps of $\p^1$ is \emph{$J$-stable} if all the repelling cycles can be followed holomorphically throughout the whole family $\Lambda$, i.e. if for all $n\geq1$, there exists $N\geq0$ and holomorphic maps $z_1,\ldots,z_N:\Lambda\rightarrow\p^1$ such that $\{z_1(\lambda),\ldots,z_N(\lambda)\}$ is exactly the set of all repelling cycles of $f_\lambda$ of exact period $n$ for all $\lambda\in\Lambda$. Note that this is equivalent to the fact that all critical points are stable~\cite{MSS}. We prove the following:

\begin{Theorem}\label{prop:totalimpliesJstable}
Let $(f,a)$ be a dynamical pair of degree $d$ of the Riemann sphere $\p^1$ parametrized by a one-dimensional complex manifold $\Lambda$. Assume that $\mathrm{Bif}(f,a)=\Lambda$. Then $f$ is $J$-stable and
\begin{itemize}
\item either $f$ is isotrivial,
\item or $J_{f_\lambda}=\p^1$ and $f_\lambda$ carries an invariant linefield for any $\lambda\in\Lambda$.
\end{itemize}
\end{Theorem}

\smallskip

The bifurcation locus of a pair $(f,a)$ is the support of natural a positive (finite) measure: the \emph{bifurcation measure} $\mu_{f,a}$ of the pair $(f,a)$, see Section \ref{sec:prelim} for a precise definition. The properties of this measure appear to be very important for studying arithmetic and dynamical properties of the pair $(f,a)$, see e.g. \cite{BD-unlikely,BD,DeMarco-heights,DeMarco-Mavraki,DWY,favregauthier,specialcubic,Conti}.
Note also that the entropy theory of dynamical pairs has been recently developed in~\cite{dTGV}.

We will say that a dynamical pair $(f,a)$ parametrized by $\Lambda$ is \emph{algebraic} if $\Lambda$ is a quasi-projective variety, if $f:\Lambda\times\p^1\to\p^1$ is a morphism and if $a:\Lambda\to\p^1$ is a rational function on $\Lambda$. An important result of DeMarco \cite{DeMarco-heights} states that any stable algebraic pair is either isotrivial or preperiodic, i.e. there exists $n>m\geq0$ such that $f_\lambda^n(a(\lambda))=f_\lambda^m(a(\lambda))$ for all $\lambda\in\Lambda$. 
In the present article, we study algebraic dynamical pairs having an absolutely continuous bifurcation measure.

\medskip

Assume that for some parameter $\lambda_0\in\Lambda$, the marked point $a$ eventually lands on a repelling periodic point $x$, that is $f_{\lambda_0}^{n}(a(\lambda_0))=x$. Let $x(\lambda)$ be the (local) natural continuation of $x$ as a periodic point of $f_\lambda$. We say that $a$ is \emph{transversely prerepelling} at $\lambda_0$ if the graphs of $\lambda\mapsto f_\lambda^n(a(\lambda))$ and $\lambda\mapsto x(\lambda)$, as subsets of $\Lambda\times\mathbb{P}^1$, are transverse at $\lambda_0$.

\medskip

Finally, recall that a rational map $f:\p^1\to\p^1$ is a \emph{Latt\`es map} if there exists an elliptic curve $E$, an endomorphism $L:E\to E$ and a finite branched cover $p:E\to\p^1$ such that $p\circ L=f\circ p$ on $E$.
Such a map has an absolutely continuous maximal entropy measure, see~\cite{Zdunik}.  In addition, when $f$ is a family of Latt\`es maps and the pair $(f,a)$ is unstable, then $\mathrm{Bif}(f,a)=\Lambda$, see e.g. \cite[\S 6]{DeMarco-Mavraki} or, e.g., Lemma~\ref{lm:bifJstab} for another proof. 

Our main result is the following.

\begin{Theorem}\label{tm:rigidLattes}
Let $(f,a)$ be an algebraic dynamical pair of $\mathbb{P}^1$ of degree $d\geq2$ parametrized by an irreducible quasi-projective curve $\Lambda$. Assume that $f$ is non-isotrivial and that $(f,a)$ is unstable. The following assertions are equivalent: 
\begin{enumerate}
\item The bifurcation locus of the dynamical pair $(f,a)$ is $\mathrm{Bif}(f,a)=\Lambda$,
\item Transversely prerepelling parameters are dense in $\Lambda$,
\item The measure $\mu_{f,a}$ is absolutely continuous with continuous Radon-Nikodym derivative outside a finite set,
\item The family $f$ is a family of Latt\`es maps.
\end{enumerate}
\end{Theorem}

Note that the hypothesis that $f$ is not isotrivial is necessary to have the equivalence between $1.$ and $4.$ (see Proposition~\ref{prop:isotrivial}).

~

The first step of the proof consists in proving that transversely prerepelling parameters are dense in the support of $\mu_{f,a}$. Using properties of Polynomial-Like Maps in higher dimension and a transversality Theorem of Dujardin for laminar currents \cite{Dujardin2012}, we prove this property holds for the appropriate bifurcation current for any tuple $(f,a_1,\ldots,a_m)$, where $f:\Lambda\times\p^k\to\p^k$ is any holomorphic family of endomorphisms of $\p^k$ and $a_1,\ldots,a_m:\Lambda\to\p^k$ are any marked points (see Theorem~\ref{tm:density}).

\medskip

As a second step, we adapt the similarity argument of Tan Lei~\cite{similarity} to show that, if $\lambda_0$ is a transversely prerepelling parameter where the bifurcation measure is absolutely continuous, the maximal entropy measure $\mu_{f_{\lambda_0}}$ of $f_{\lambda_0}$ is also non-singular with respect to the Fubini-Study form on $\p^1$. As Zdunik \cite{Zdunik} has shown, this implies $f_{\lambda_0}$ is a Latt\`es map.

This gives, in particular, the following.
\begin{Theorem}\label{tm:rigidlattes-dim1}
Fix an  integer $d\geq2$ and let $(f,a)$ be a holomorphic dynamical pair of degree $d$ of $\p^1$ parametrized by a K\"ahler manifold $(M,\omega)$ of dimension $1$. Assume the support of $\mu_{f,a}$ is $\supp(\mu_{f,a})=M$. Then, the following are equivalent:
\begin{enumerate}
\item the measure $\mu_{f,a}$ is absolutely continuous with respect to $\omega$ and the Radon-Nikodym derivative $\frac{d\mu_{f,a}}{d\omega}$ is continuous outside an analytic subvariety of $M$,
\item the family $f$ is a family of Latt\`es maps.
\end{enumerate}
\end{Theorem}

We can see Theorem \ref{tm:rigidLattes} as a partial parametric counterpart of Zdunik's result. However, the comparison with Zdunik's work ends there: Rational maps with $\p^1$ as a Julia sets are, in general, not Latt\`es maps. Indeed, Latt\`es maps form a strict subvariety of the space of all degree $d$ rational maps, and maps with $J_f=\p^1$ form a set of positive volume by \cite{Rees}. In a way, Theorem \ref{tm:rigidLattes} is a stronger rigidity statement that the dynamical one. 

Note also that we only use the fact that $\Lambda$ is a quasi-projective curve to prove the equivalence between $\mathrm{Bif}(f,a)=\Lambda$ and the smoothness of the bifurcation measure, relying on \cite{McM-stable}. We don't know how to get rid of this algebraicity assumption, without using the No Invariant Line Field Conjecture of McMullen, which is far from being proved.

~

Recall that, as in dimension $1$, an endomorphism $f$ of $\p^{k}$ is a \emph{Latt\`es map} if there exists an abelian variety $A$, a finite branched cover $p:A\to\p^{k}$ and an isogeny $I:A\to A$ such that $p\circ I=f\circ p$ on $A$.
Berteloot and Loeb \cite{BL} and then Berteloot and Dupont \cite{BertelootDupont} generalized Zdunik's work to endomorphisms of $\p^{k}$: $f$ is a Latt\`es map of $\p^{k}$ if and only if the measure $\mu_f$ is is not singular with respect to $\omega_{\p^{k}}^k$, see also \cite{DupontThese}. 
Recall finally that a repelling periodic point of $f$ is $J$-\emph{repelling} if it belongs to $\mathrm{supp}(\mu_f)$.

As an important part of our arguments applies in any dimension, we have the following higher dimensional counterpart to Theorem~\ref{tm:rigidlattes-dim1}.

\begin{Theorem}\label{tm:rigidlattes}
Fix integers $d\geq2$ and $k\geq1$ and let $(f,a)$ be any holomorphic dynamical pair of degree $d$ of $\p^k$ parametrized by a K\"ahler manifold $(M,\omega)$ of dimension $k$. Assume that for all $\lambda\in M$, any $J$-repelling periodic point of $f_\lambda$ is linearizable.
Assume in addition that $\mu_{f,a}:=T_{f,a}^k$ satisfies $\supp(\mu_{f,a})=M$. Then the following are equivalent:
\begin{enumerate}
\item the measure $\mu_{f,a}$ is absolutely continuous with respect to $\omega^k$ and $\frac{d\mu_{f,a}}{d\omega^k}$ is continuous outside an analytic subvariety of $M$,
\item the family $f$ is a family of Latt\`es maps of $\p^k$.
\end{enumerate}
\end{Theorem}

The paper is organized as follows. In section~\ref{sec:prelim}, we recall the construction of the bifurcation currents of marked points and properties of Polynomial-Like Maps. Section~\ref{sec:support} is dedicated to proving the density of transversely prerepelling parameters. In section~\ref{sec:similar}, we establish the similarity property for the bifurcation and maximal entropy measures. Finally, in section~\ref{sec:proof} we prove Theorems \ref{prop:totalimpliesJstable}, \ref{tm:rigidLattes}, \ref{tm:rigidlattes-dim1} and \ref{tm:rigidlattes} and list related questions.

\paragraph*{Acknowledgements}
I would like to thank Charles Favre and Gabriel Vigny whose interesting discussions, remarks and questions led to an important part of this work. This research is partially supported by the ANR grant Fatou ANR-17-CE40-0002-01.
I also would like to thank the anonymous referee for helpful comments and remarks.

\section{Dynamical preliminaries}\label{sec:prelim}

 \subsection{The bifurcation current of a dynamical tuple}\label{background}
For this section, we follow the presentation of \cite{favredujardin,Dujardin2012}. Even though everything is presented in the case $k=1$ and for marked \emph{critical} points, the exact same arguments give what we present below.

\smallskip

Let $\Lambda$ be a complex manifold and let $f:\Lambda \times \p^k\to\p^k$ be a holomorphic family of endomorphisms of $\p^k$ of algebraic degree $d \geq 2$: $f$ is holomorphic and $f_\lambda:=f(\lambda,\cdot):\p^{k}\to\p^{k}$ is an endomorphism of algebraic degree $d$.

\begin{definition}
Fix integers $m\geq1$, $d\geq2$ and let $\Lambda$ be a complex manifold. A \emph{dynamical $(m+1)$-tuple} $(f,a_1,\ldots,a_m)$ of $\p^{k}$ of degree $d$ parametrized by $\Lambda$ is a holomorphic family $f$ of endomorphisms of $\p^{k}$ of degree $d$ parametrized by $\Lambda$, endowed with $m$ holomorphic maps \emph{(marked points)} $a_1,\ldots,a_m:\Lambda\to\p^{k}$.
\end{definition}

Let $\omega_{\p^k}$ be the standard Fubini-Study form on $\p^k$ and $\pi_\Lambda:\Lambda\times\p^k\to\Lambda$ and $\pi_{\p^k}:\Lambda\times\p^k\to\p^k$ be the canonical projections. Finally, let $\widehat{\omega}:=(\pi_{\p^k})^*\omega_{\p^k}$. A family $f:\Lambda \times \p^k\to\p^k$ naturally induces a fibered dynamical system $\hat{f}:\Lambda \times \p^k\to\Lambda\times\p^k$, given by $\hat{f}(\lambda,z):=(\lambda,f_\lambda(z))$. It is known that
the sequence $d^{-n}(\hat{f}^n)^*\widehat{\omega}$ converges to a closed positive $(1,1)$-current $\widehat{T}$ on $\Lambda\times\p^k$ with continuous potential. Moreover, for any $1\leq j\leq k$,
\[\hat{f}^*\widehat{T}^j=d^j\cdot \widehat{T}^{j}\]
and $\widehat{T}^k|_{\{\lambda_0\}\times\p^1}=\mu_{\lambda_0}$ is the unique measure of maximal entropy $k\log d$ of $f_{\lambda_0}$ for all $\lambda_0\in\Lambda$.

For any $n\geq1$, we have $\widehat{T}=d^{-n}(\hat{f}^{n})^*\hat{\omega}+d^{-n}dd^c\widehat{u}_n$, where $(\widehat{u}_n)_n$ is a locally uniformly bounded sequence of continuous functions.

\medskip

Pick now a dynamical $(m+1)$-tuple $(f,a_1,\ldots,a_m)$ of degree $d$ of $\p^k$. Let $\Gamma_{a_j}\subset\Lambda\times\p^k$ be the graph of the map $a_j$ and set
 \[\mathfrak{a}:=(a_1,\ldots,a_m).\]
 \begin{definition}
For $1\leq i\leq m$, the \emph{bifurcation current} $T_{f,a_i}$ of the pair $(f,a_i)$ is the closed positive $(1,1)$-current on $\Lambda$ defined by
 \[T_{f,a_i}:=(\pi_{\Lambda})_*\left(\widehat{T}\wedge[\Gamma_{a_j}]\right)\]\\
 and we define the \emph{bifurcation current} $T_{f,\mathfrak{a}}$ of the $(m+1)$-tuple $(f,a_1,\ldots,a_m)$ as
 \[T_{f,\mathfrak{a}}:=T_{f,a_1}+\cdots+T_{f,a_k}.\]
 \end{definition}

For any $\ell\geq0$, write
\[\mathfrak{a}_\ell(\lambda):=\left(f_\lambda^{\ell}(a_1(\lambda)),\ldots,f_\lambda^{\ell}(a_m(\lambda))\right), \ \lambda\in\Lambda.\]
Let now $K\subset\Lambda$ be a compact subset of $\Lambda$ and let $\Omega$ be some relatively compact neighborhood of $K$, then $(a_{\ell})^*(\omega_{\p^{k}})$ is bounded in mass in $\Omega$ by $C d^{\ell}$, where $C$ depends on $\Omega$ but not on $\ell$.

Note that the proof of \cite[Proposition-Definition~3.1 and Theorem~3.2]{favredujardin} (which is for marked critical points and when $k=1$) works similarly when $k>1$ and $a$ is non-critical. Applying verbatim their proof, we have the following, see also the proof of \cite[Theorem~9.1]{DeMarco2} which adapts also perfectly here.

\begin{lemma}\label{lm:DF}
For any $1\leq i\leq k$, the support of $T_{f,a_i}$ is the set of parameters $\lambda_0\in\Lambda$ such that the sequence $\{\lambda\mapsto f_\lambda^n(a_i(\lambda))\}$ is not a normal family at $\lambda_0$.

Moreover, writing $a_{i,\ell}(\lambda):=f_{\lambda}^\ell(a_i(\lambda))$, there exists a locally uniformly bounded family $(u_{i,\ell})$ of continuous functions on $\Lambda$ such that 
 \[(a_{i,\ell})^*(\omega_{\p^{k}})=d^{\ell} T_{f,a_i}+dd^cu_{i,\ell} \ \text{on} \ \Lambda.\] 
\end{lemma}

As a consequence, for all $j\geq1$, we have
\[(a_{i,\ell})^*(\omega_{\p^{k}}^j)=d^{j\ell} T_{f,a_i}^j+\sum_{s=1}^j\binom{j}{s} d^{\ell(j-s)}\cdot(dd^cu_{i,\ell})^s\wedge T_{f,a_i}^{j-s},\]
so that the mass of the $(j,j)$-current $(a_{i,\ell})^*(\omega_{\p^{k}}^j)-d^{j\ell} T_{f,a_i}^j$ is $O(d^{(j-1)\ell})$ on compact subsets of $\Lambda$. In particular, one sees that
\begin{align}
T_{f,a_i}^{k+1}=0 \ \text{on} \ \Lambda. \label{nointersection>q}
\end{align}

\medskip 
 
Let us still denote  $\pi_\Lambda:\Lambda \times(\p^k)^m\to \Lambda$ be the projection onto the first coordinate and for $1\leq i\leq k$, let $\pi_i:\Lambda \times(\p^k)^m\to \Lambda\times\p^k$ be the projection onto $\Lambda$ times the $i$-th factor of the product $(\p^k)^m$.  
Finally, we denote by $\Gamma_{\mathfrak{a}}$ the graph of $\mathfrak{a}$:
 \[\Gamma_{\mathfrak{a}}:= \{ (\lambda,z_1,\ldots,z_m), \ \forall j, z_j= a_j(\lambda)\} \subset \Lambda \times (\p^k)^m.\] 
Following verbatim the proof of~\cite[Lemma 2.6]{AGMV}, we get
\[\frac{1}{(mk)!}T_{f,\mathfrak{a}}^{mk}=\bigwedge_{\ell=1}^mT_{f,a_\ell}^{k}= (\pi_\Lambda)_* \left( \bigwedge_{i=1}^m \pi_i^*\left(\hat{T}^{k}\right) \wedge \left[\Gamma_{\mathfrak{a}}\right]\right).\]

\subsection{Hyperbolic sets supporting a PLB ergodic measure}\label{sec:PLM}

\begin{definition}
Let $W\subset \C^k$ be a bounded open set. We say that a positive measure $\nu$ compactly supported on $W$ is \emph{PLB} if the psh functions on $W$ are integrable with respect to $\nu$.
\end{definition}

We aim here at proving the following proposition in the spirit of \cite[Lemma~4.1]{Dujardin2012}:

\begin{proposition}\label{goodhyperbolicset}
Let $f:\mathbb{P}^k\rightarrow\mathbb{P}^k$ be an endomorphism of degree $d\geq2$.
There exists a small ball $\mathbb{B}\subset\p^k$, an integer $m\geq1$, a $f^m$-invariant compact set $K\subset\mathbb{B}$ and an integer $N\geq2$ such that
\begin{itemize}
\item $f^m|_K$ is uniformly expanding and repelling periodic points of $f^m$ are dense in $K$,
\item there exists a unique probability measure $\nu$ supported on $K$ such that $(f^m|_K)^*\nu=N\nu$ which is PLB.
\end{itemize}
\end{proposition}

Even though this result is considered folklore, we include a proof relying on properties of  \emph{polynomial-like map}. We refer to \cite{DS-PLM} for more about polynomial-like maps.
Given an complex manifold $M$ and an open set $V\subset M$, we say that $V$ is \emph{$S$-convex} if there exists a continuous strictly plurisubharmonic function  on $V$. In fact, this implies that there exists a smooth strictly psh function $\psi$, whence there exists a K\"ahler form $\omega:=dd^c\psi$ on $V$.

\begin{definition}
Given a connected $S$-convex open set and a relatively compact open set $U\subset V$, a map $f:U\rightarrow V$ is \emph{polynomial-like} if $f$ is holomorphic and proper.
\end{definition}

The \emph{filled-Julia set} of $f$ is the set
\[\mathcal{K}_f:=\bigcap_{n\geq0}f^{-n}(U).\]
The set $\mathcal{K}_f$ is full, compact, non-empty and it is the largest totally invariant compact subset of $V$, i.e. such that $f^{-1}(\mathcal{K}_f)=\mathcal{K}_f$.

\medskip

The topological degree $d_t$ of $f$ is the number of preimages of any $z\in V$ by $f$, counted with multiplicity. Let $k:=\dim V$. We define
\[d_{k-1}^*:=\sup_\varphi\left\{d_t\cdot \limsup_{n\rightarrow\infty}\|\Psi^ndd^c\varphi\|_U^{1/n}\, ; \ \varphi \text{ is psh on }V\right\},\]
where $\Psi:=d_t^{-1}f_*$. According to Theorem 3.2.1 and Theorem~3.9.5 of \cite{DS-PLM}, we have the following.

\begin{theorem}[Dinh-Sibony]\label{tm:DSPLM}
Let $f:U\to V$ be a polynomial-like map of topological degree $d_t\geq2$. There exists a unique probability measure $\mu$ supported by $\partial \mathcal{K}_f$ which is ergodic and such that
\begin{enumerate}
\item for any volume form $\Omega$ of mass $1$ in $L^2(V)$,  one has $d_t^{-n}(f^n)^*\Omega\rightarrow \mu$ as $n\to\infty$,
\item if $d_{k-1}^*<d_t$, the measure $\mu$ is PLB and repelling periodic points are dense in $\supp(\mu)$.
\end{enumerate}
\end{theorem}

\begin{proof}[Proof of Proposition~\ref{goodhyperbolicset}]
The first argument is an inverse branches argument which follows Briend-Duval~\cite[Section~3]{briendduval}. Let $B:=\B(x,\epsilon)$ be a small ball around  a $\mu_f$-generic point $x$. Since $\mu_f$ is mixing, we have $\mu_f(f^{-n}(B) \cap B) \simeq \mu(B)^2$ for $n$ large enough. In particular, using $(f^{n})^*\mu_f=d^{nk}\mu_f$, we deduce there exists $C>0$ such that $f^n$ has $M(n)\geq Cd^{nk}$ inverse branches $g_1,\ldots,g_{M(n)}$ defined on $B$ with
\begin{itemize}
\item $g_i(B)\subset B$  and $g_i$ is uniformly contracting on $B$ for all $i$,
\item $g_i(B)\cap g_j(B)=\emptyset$ for all $i\neq j$.
\end{itemize}
Fix $m\geq n_0$ large enough so that $Cd^{mk}>d^{(k-1)m}\geq2$ and set
\[V:=B, \ \ U:=\bigcup_{j=1}^{M(m)}g_j(B), \ \ N:=M(m) \ \ \text{and} \ \ g:=f^m|_U.\] 
The map $g:U\to V$ is polynomial-like of topological degree $N$, whence its equilibrium measure $\nu$ is the unique probability measure which satisfies $g^*\nu=N\nu$ by the first part of Theorem \ref{tm:DSPLM}. We let $K:=\supp(\nu)$. Since the $g_i$'s are uniformly contracting, the compact set $K$ is $f^m$-hyperbolic.

To conclude, it is sufficient to verify that $N>d_{k-1}^{*}$. Fix $n\geq1$ and $\varphi$ psh on $V$. Let $\omega$ be the (normalized) restriction to $V$ of the Fubini-Study form of $\mathbb{P}^{k}$. Then, since $\Psi=\frac{1}{N}g_*$,
\begin{align*}
\|\Psi^{n}(dd^{c}\varphi)\|_U & =\int_U\left(\Psi^{n}(dd^{c}\varphi)\right)\wedge \omega^{k-1}=\int_U\frac{1}{N^{n}}\left((g^{n})_*(dd^{c}\varphi)\right)\wedge \omega^{k-1}\\
& = \frac{1}{N^{n}}\int_Udd^{c}\varphi\wedge (g^{n})^{*}\omega^{k-1}\\
&=\frac{1}{N^{n}}\int_Udd^{c}\varphi\wedge (d^{mn}\omega+dd^{c}u_{nm})^{k-1}
\end{align*}
where $(u_{n})_n$ is a uniformly bounded sequence of continuous functions on $\mathbb{P}^k$. In particular, by the Chern-Levine-Niremberg inequality, if $U\subset W\subset V$, there exists a constant $C'>0$ depending only on $W$ such that
\begin{align*}
\|\Psi^{n}(dd^{c}\varphi)\|_U & =\left(\frac{d^{(k-1)m}}{N}\right)^n\int_Udd^{c}\varphi\wedge (\omega+d^{-nm}dd^{c}u_{nm})^{k-1}\\
& \leq \left(\frac{d^{(k-1)m}}{N}\right)^n C'\|dd^c\varphi\|_W.
\end{align*}
Taking the $n$-th root and passing to the limit, we get
\[\frac{d_{k-1}^*}{N}\leq \frac{d^{(k-1)m}}{N}<1\]
by assumption. The second part of Theorem~\ref{tm:DSPLM} allows us to conclude.
\end{proof}

\section{The support of bifurcation currents}\label{sec:support}

Pick a complex manifold $\Lambda$ and let $m,k\geq1$ be so that $\dim\Lambda\geq km$. Let $(f,a_1,\ldots,a_m)$ be a dynamical $(m+1)$-tuple  of $\p^k$ of degree $d$ parametrized by $\Lambda$.

\begin{definition}
We say that the marked points $a_1,\ldots,a_m$ are \emph{transversely $J$-prerepelling (resp. properly $J$-prerepelling)} at a parameter $\lambda_0$ if there exists integers $n_1,\ldots,n_m\geq1$ such that $f_{\lambda_0}^{n_j}(a_j(\lambda_0))=z_j$ is a repelling periodic point of $f_{\lambda_0}$ and, if $z_j(\lambda)$ is the natural continuation of $z_j$ as a repelling periodic point of $f_\lambda$ in a neighborhood $U$ of $\lambda_0$, such that
\begin{enumerate}
\item $z_j(\lambda)\in J_{\lambda}$ for all $\lambda\in U$ and all $1\leq j\leq m$,
\item the graphs of $A:\lambda\mapsto(f_\lambda^{q_1}(a_1(\lambda)),\ldots,f_\lambda^{q_m}(a_m(\lambda)))$ and of $Z:\lambda\mapsto(z_1(\lambda),\ldots,z_m(\lambda))$ intersect transversely (resp. along an analytic subset of $\Lambda\times\p^k$ of codimension $km$) at $\lambda_0$.
\end{enumerate}
\end{definition}

In this section, we prove the following:

\begin{theorem}\label{tm:density}
Let $(f,a_1,\ldots,a_m)$ be a dynamical $(m+1)$-tuple  of $\p^k$ of degree $d$ parametrized by $\Lambda$ with $km\leq \dim\Lambda$.

Then the support of $T_{f,a_1}^k\wedge\cdots\wedge T_{f,a_m}^k$ coincides with the closure of the set of parameters at which $a_1,\ldots,a_m$ are transversely $J$-prerepelling.
\end{theorem}

\begin{remark}\normalfont
The hypothesis on the dimension of the parameter space looks a priori artificial, but transversely $J$-prerepelling parameters form analytic subsets of codimension $km$. In particular, it is not clear to me that you can prove the existence (and thus the Zariski density) of such parameters if $\dim \Lambda<km$.
\end{remark}

\subsection{Properly prerepelling marked points bifurcate}

First, we give a quick proof of the fact that properly $J$-prerepelling parameters belong to the support of $T_{f,a_1}^k\wedge \cdots \wedge T_{f,a_m}^k$, without any additional assumption.

\begin{theorem}\label{tm:densitypart1}
Let $(f,a_1,\ldots,a_m)$ be a dynamical $(m+1)$-tuple  of $\p^k$ of degree $d$ parametrized by $\Lambda$ with $km\leq \dim\Lambda$. Pick any parameter $\lambda_0\in\Lambda$ such that $a_1,\ldots,a_m$ are properly $J$-prerepelling at $\lambda_0$. Then $\lambda_0\in\supp\left(T_{f,a_1}^k\wedge\cdots\wedge T_{f,a_m}^k\right)$.
\end{theorem}

The proof of this result is an adaptation of the strategy of Buff and Epstein~\cite{buffepstein} and the strategy of Berteloot, Bianchi and Dupont~\cite{BBD}, see also~\cite{Article1,AGMV}.
Since it follows closely that of~\cite[Theorem~B]{AGMV}, we shorten some parts of the proof.

\medskip

  Before giving the proof of Theorem~\ref{tm:densitypart1}, remark that our properness assumption is equivalent to saying that the local hypersurfaces
\[X_j:=\{\lambda\in\Lambda\, ; \ f_\lambda^{q_j}(a_j(\lambda))=z_j(\lambda)\}\]
intersecting at $\lambda_0$ satisfy $\mathrm{codim}\left(\bigcap_jX_j\right)=km$.

\begin{proof}[Proof of Theorem~\ref{tm:densitypart1}]
According to \cite[Lemma 6.3]{Article1}, we can reduce to the case when $\Lambda$ is an open set of $\C^{km}$. 
Take a small ball $B$ centered at $\lambda_0$ in $\Lambda$.
Up to reducing $B$, we can assume $z_j(\lambda)$ can be followed as a repelling periodic point of $f_\lambda$ for all $\lambda\in B$. Up to reducing $B$, our assumption is equivalent to the fact that $\bigcap_jX_j=\{\lambda_0\}$.

~

We let $\mu:=T_{f,a_1}^k\wedge\cdots\wedge T_{f,a_m}^k$. Our aim here is to exhibit a basis of neighborhood $\{\Omega_n\}_n$ of $\lambda_0$ in $\B$ with $\mu(\Omega_n)>0$ for all $n$.
For any $m$-tuple $\un:=(n_1,\ldots,n_m)\in(\mathbb{N}^*)^m$, we let
\begin{eqnarray*}
F_{\un}:\Lambda\times(\p^k)^m & \longrightarrow & \Lambda\times(\p^k)^m\\
(\lambda,z_1,\ldots,z_m) & \longmapsto & (\lambda,f^{n_1}_\lambda(z_1),\ldots,f^{n_m}_\lambda(z_m))~.
\end{eqnarray*}
For a $m$-tuple $\un=(n_1,\ldots,n_m)$ of positive integers, we set 
\[|\un|:=n_1+\cdots+n_m~.\]
We also denote
\[\mathfrak{A}_{\un}(\lambda):=\left(f_\lambda^{n_1}(a_1(\lambda)),\ldots,f_\lambda^{n_m}(a_m(\lambda))\right), \ \lambda\in\Lambda.\]
As in~\cite{AGMV}, we have the following.
\begin{lemma}\label{lm:formula2}
For any $m$-tuple $\un=(n_1,\ldots,n_m)$ of positive integers, we let $\Gamma_{\un}$ be the graph in $\Lambda\times(\p^k)^m$ of $\mathfrak{A}_{\un}$. Then, for any Borel set $B\subset\Lambda$, we have
\begin{eqnarray*}
\mu(B) & = & d^{-k\cdot|\un|}\int_{B\times(\p^k)^m}\left(\bigwedge_{j=1}^m(\pi_j)^*\left(\widehat{T}^k\right)\right)\wedge\left[\Gamma_{\un}\right]~.
\end{eqnarray*}
\end{lemma}

\medskip

Suppose that the point $z_j$ is $r_j$-periodic. For the sake of simplicity, we let in the sequel $\mathfrak{A}_{n}:=\mathfrak{A}_{\underline{q}+n\underline{r}}$, where $\underline{q}=(q_1,\ldots,q_m)$, $\underline{r}=(r_1,\ldots,r_m)$ are given as above and $\underline{q}+n\underline{r}=(q_1+nr_1,\ldots,q_m+nr_m)$. Again as above, we let $\Gamma_n$ be the graph of $\mathfrak{A}_n$.

~

\par\noindent Let $z:=(z_1,\ldots,z_m)$ and fix any small open neighborhood $\Omega$ of $\lambda_0$ in $\Lambda$.
Set
\[I_n:=\int_{\Omega\times(\p^k)^m}\left(\bigwedge_{j=1}^m(\pi_j)^*\left(\widehat{T}^k\right)\right)\wedge\left[\Gamma_{n}\right].\]
Since $z_j(\lambda)$ is repelling and periodic for $f_\lambda$ for all $\lambda\in B$ (if $B$ has been chosen small enough), there exists a constant $K>1$ such that
\[d_{\p^k}(f_\lambda^{r_j}(z),f_\lambda^{r_j}(w))\geq K\cdot d_{\p^k}(z,w)\]
for all $z,w\in\B(z_j(\lambda_0),\epsilon)$ and all $\lambda\in B$ for some given $\epsilon>0$. In particular, if $S_n$ is the connected component of $\Gamma_{n}\cap \Lambda\times\B_\epsilon^m(z)$ containing $(\lambda_0,z)$, the current $[S_n]$ is vertical-like in in $\Lambda\times\B_\epsilon^m(z)$ and there exists $n_0\geq1$ and a basis of neighborhood $\Omega_n$ of $\lambda_0$ in $\Lambda$ such that
\[\supp([S_n])=S_n\subset\Omega_{n}\times\B_\epsilon^m(z),\]
for all $n\geq n_0$.

~

 Let $S$ be any weak limit of the sequence $[S_n]/\|[S_n]\|$. Then $S$ is a closed positive $(mk,mk)$-current of mass $1$ in $B\times\B^m_\epsilon(z)$ with $\supp(S)\subset\{\lambda_0\}\times\B_\epsilon^m(z)$. Hence $S=M\cdot[\{\lambda_0\}\times\B^m_\epsilon(z)]$, where $M^{-1}>0$ is the volume of $\B_\epsilon^m(z)$ for the volume form $\bigwedge_j(\omega_j^k)$, where $\omega_j=(p_j)^*\omega_{\p^{k}}$ and $p_j:(\p^k)^m\to\p^k$ is the projection on the $j$-th coordinate.

As a consequence, $[S_n]/\|[S_n]\|$ converges weakly to $S$ as $n\to\infty$ and, since the $(mk,mk)$-current $\bigwedge_{j=1}^m(\pi_j)^*(\widehat{T}^k)$ is the wedge product of $(1,1)$-currents with continuous potentials,  we have
\[\bigwedge_{j=1}^m(\pi_j)^*\left(\widehat{T}^k\right)\wedge\frac{[S_n]}{\|[S_n]\|}\longrightarrow \bigwedge_{j=1}^m(\pi_j)^*\left(\widehat{T}^k\right)\wedge S\]
as $n\to+\infty$. Whence
\begin{align*}
\liminf_{n\to\infty}\left(\|[S_n]\|^{-1}\cdot I_n\right)& \geq\liminf_{n\to\infty}\int\bigwedge_{j=1}^m(\pi_j)^*\left(\widehat{T}^k\right)\wedge\frac{[S_n]}{\|[S_n]\|}\\
&\geq \int\bigwedge_{j=1}^m(\pi_j)^*\left(\widehat{T}^k\right)\wedge S.
\end{align*}
By the above, this gives
\[\liminf_{k\to\infty}\left(\|[S_n]\|^{-1}\cdot I_n\right)\geq M\cdot \int[\{\lambda_0\}\times\B_\epsilon^m(z)]\wedge\bigwedge_{j=1}^m(\pi_j)^*\left(\widehat{T}^k\right)~,\]
In particular, there exists $n_2\geq n_1$ such that for all $n\geq n_2$,
\[\|[S_n]\|^{-1}\cdot I_n\geq \frac{M}{2}\cdot\int[\{\lambda_0\}\times\B_\epsilon^m(z)]\wedge \bigwedge_{j=1}^m(\pi_j)^*\left(\widehat{T}^k\right)~.\]
Finally, since $[S_n]$ is a vertical-like current, up to reducing $\epsilon>0$, Fubini Theorem gives
\[\liminf_{n\to\infty}\|[S_n]\|\geq \prod_{j=1}^m\int_{\B(z_j,\epsilon)}\omega_{\textup{FS}}^k\geq \left(c\cdot \epsilon^{2k}\right)^m>0~.\]
Up to increasing $n_0$, we may assume $\|[S_n]\|\geq (c\epsilon^{2k})^m/2$ for all $n\geq n_0$. Letting $\alpha=M(c\epsilon^{2k})^m/4>0$, we find
\[\int_{\Omega\times(\p^k)^m}\left(\bigwedge_{j=1}^m(\pi_j)^*\left(\widehat{T}^k\right)\right)\wedge\left[\Gamma_{n}\right]\geq \alpha\int\left[\{\lambda_0\}\times\B_\epsilon^m(z)\right]\wedge\bigwedge_{j=1}^m(\pi_j)^*\left(\widehat{T}^k\right).\]

To conclude the proof of Theorem~\ref{tm:densitypart1}, we rely on the following purely dynamical result, which is an immediate adaptation of~\cite[Lemma~3.5]{AGMV}.
\begin{lemma}\label{lm:limit}
For any $\delta>0$ and $x=(x_1,\ldots,x_m)\in\left(\supp(\mu_{\lambda_0})\right)^m$, we have
\begin{eqnarray*}
\int\left[\{\lambda_0\}\times\B_\delta^m(x)\right]\wedge\bigwedge_{j=1}^m(\pi_j)^*\left(\widehat{T}^k\right)=\prod_{j=1}^m\mu_{\lambda_0}(\B(x_j,\delta))>0.
\end{eqnarray*} 
\end{lemma}

\medskip

We can now conclude the proof of Theorem \ref{tm:densitypart1}. 
Pick any open neighborhood $\Omega$ of $\lambda_0$ in $\Lambda$. By the above and Lemma \ref{lm:limit}, we have an integer $n_0\geq1$ and constants $\alpha,\epsilon>0$ such that for all $n\geq n_0$,
\begin{eqnarray*}
\mu(\Omega)\geq \alpha\cdot d^{-k\left(|\underline{q}|+n|\underline{r}|\right)}\prod_{j=1}^m\mu_f(\mathbb{B}(z_j,\epsilon))>0~.
\end{eqnarray*}
In particular, this yields $\mu(\Omega)>0$. By assumption, this holds for a basis of neighborhoods of $\lambda_0$ in $\Lambda$, whence we have $\lambda_0\in\textup{supp}(\mu)$.
\end{proof}

\subsection{Density of transversely prerepelling parameters}
To finish the proof of Theorem~\ref{tm:density}, it is sufficient to prove that any point of the support of $T_{f,a_1}^{k}\wedge\cdots\wedge T_{f,a_m}^{k}$ an be approximated by transversely $J$-prerepelling parameters. We follow the strategy of the proof of Theorem~0.1 of \cite{Dujardin2012} to establish this approximation property. Precisely, we prove here the following.

\begin{theorem}\label{tm:densitypart2}
Let $(f,a_1,\ldots,a_m)$ be a dynamical $(m+1)$-tuple  of $\p^k$ of degree $d$ parametrized by $\Lambda$ with $km\leq \dim\Lambda$.

 Then, any parameter $\lambda\in\Lambda $ lying in the support of the current $T_{f,a_1}^k\wedge\cdots\wedge T_{f,a_m}^k$ can be approximated by parameters at which $a_1,\ldots,a_m$ are transversely $J$-prerepelling.
\end{theorem}

We rely on the following property of \emph{PLB} measures (see \cite{DS-PLM}):

\begin{lemma}\label{lm:PLB}
Let $\nu$ be PLB with compact support in a bounded open set $W\subset\C^k$ and let $\psi$ be a psh function on $\C^k$. The function $G_\psi$ defined by
\[G_\psi(z):=\int\psi(z-w)d\nu(w), \ z\in\C^k,\]
is psh and locally bounded on $\C^k$.
\end{lemma}

\begin{proof}[Proof of Theorem \ref{tm:densitypart2}]
 We follow the strategy of the proof of \cite[Theorem 0.1]{Dujardin2012}.
 Write $\mu:=T^k_{f,a_1}\wedge\cdots\wedge T_{f,a_m}^k$
and pick $\lambda_0\in\supp(\Omega)$.

\smallskip

According to Proposition~\ref{goodhyperbolicset}, there exists an integer $m\geq1$ and a $f_{\lambda_0}^m$-compact set $K\subset\p^k$ contained in a ball and $N\geq2$ such that 
\begin{itemize}
\item $f_{\lambda_0}^m|_K$ is uniformly hyperbolic and repelling periodic points of $f_{\lambda_0}^m$ are dense in $K$,
\item there exists a unique probability measure $\nu$ supported on $K$ such that $(f_{\lambda_0}^m|_K)^*\nu=N\nu$ which is PLB.
\end{itemize}
Since $K$ is hyperbolic, there exists $\epsilon>0$ and a unique holomorphic motion $h:\B(\lambda_0,\epsilon)\times K\to\p^k$ which conjugates the dynamics, i.e. $h$ is continuous and such that
\begin{itemize}
\item for all $\lambda\in\B(\lambda_0,\epsilon)$, the map $h_\lambda:=h(\lambda,\cdot):K\to\p^k$ is injective and $h_{\lambda_0}=\mathrm{id}_K$,
\item for all $z\in K$, the map $\lambda\in\B(\lambda_0,\epsilon)\mapsto h_\lambda(z)\in\p^k$ is holomorphic, and
\item for all $(\lambda,z)\in \B(\lambda_0,\epsilon)\times K$, we have $h_\lambda\circ f_{\lambda_0}^m(z)=f_\lambda^m\circ h_\lambda(z)$,
\end{itemize}
see e.g. \cite[Theorem 2.3 p. 255]{demelovanstrien} or \cite[Appendix~A.1]{BBD}.

For all $z:=(z_1,\ldots,z_m)\in K^m$, we denote by $\Gamma_z$ the graph of the holomorphic map $\lambda\mapsto (h_\lambda(z_1),\ldots,h_\lambda(z_m))$.

\medskip

We define a closed positive $(km,km)$-current on $\B(\lambda_0,\epsilon)\times(\p^k)^m$ by letting
\[\widehat{\nu}:=\int_{K^m}[\Gamma_z]\mathrm{d}\nu^{\otimes m}(z),\]
where $\Gamma_z=\{(\lambda,h_\lambda(z_1),\ldots,h_\lambda(z_m))\, ; \ \lambda\in\B(\lambda_0,\epsilon)\}$ for all $z=(z_1,\ldots,z_m)\in K^m$.

\begin{claim}
There exists a $(km-1,km-1)$-current $U$ on $\mathbb{B}(\lambda_0,\epsilon)\times(\p^{k})^{m}$ which is locally bounded and such that $\hat{\nu}=dd^cU$.
\end{claim}

Recall that we have set $\mathfrak{a}_n(\lambda):=(f_\lambda^n(a_1(\lambda)),\ldots,f_\lambda^n(a_m(\lambda)))$. We define $\mathfrak{a}_n^*\hat{\nu}$ by
\[\mathfrak{a}_n^*\hat{\nu}:=(\pi_{1})_*\left(\hat{\nu}\wedge [\Gamma_{\mathfrak{a}_n}]\right),\]
where $\pi_1:\B(t_0,\epsilon)\times(\p^k)^m\to\B(t_0,\epsilon)$ is the canonical projection onto the first coordinate.
According to the claim, locally we have $\hat{\nu}=dd^cU$, for some bounded $(km-1,km-1)$-current $U$. In particular, we get $\mathfrak{a}_n^*\hat{\nu}=\mathfrak{a}_n^*(dd^cU)$, as expected.

Let $\omega$ be the Fubini-Study form of $\p^k$ and $\hat{\Omega}:=(\pi_2)^*(\omega^{k}\otimes\cdots\otimes\omega^k)$, where $\pi_2:\B(\lambda_0,\epsilon)\times(\p^k)^m\to(\p^k)^m$ is the canonical projection onto the second coordinate. Then
\[\hat{\nu}-\hat{\Omega}=dd^cV\]
where $V$ is bounded on $\B(\lambda_0,\epsilon)\times(\p^k)^m$, hence 
\[d^{-kmn}\mathfrak{a}_n^*(\hat{\nu})-d^{-kmn}\mathfrak{a}_n^*(\hat{\Omega})=d^{-kmn}\mathfrak{a}_n^*(dd^cV).\]
On the other hand, we have $\frac{1}{d^{km}}\hat{f}^*(\hat{\Omega})=\hat{\Omega}+dd^cW$, where $W$ is bounded on $\B(\lambda_0,\epsilon)\times(\p^k)^m$, hence $\frac{1}{d^{kmn}}(\hat{f}^n)^*(\hat{\Omega})=\hat{\omega}+dd^cW_n$, where $W_n-W_{n+1}=O(d^{-n})$.
In particular, $\frac{1}{d^n}(\hat{f}^n)^*(\hat{\Omega})\wedge [\Gamma_a]=d^{-n}\left(\hat{\Omega}\wedge [\Gamma_{\mathfrak{a}_n}]\right)+dd^cO(d^{-n})$, 
hence $\mu=\lim_nd^{-kmn}\mathfrak{a}_n^*(\hat{\Omega})$. This yields 
\[\lim_{n\to\infty}d^{-nkm}(\pi_{1})_*\left(\hat{\nu}\wedge [\Gamma_{\mathfrak{a}_n}]\right)=\mu.\]

\smallskip

We now use \cite[Theorem 3.1]{Dujardin2012}: as $(2km,2km)$-currents on $\B(\lambda_0,\epsilon)\times(\p^k)^m$,
\[\hat{\nu}\wedge [\Gamma_{\mathfrak{a}_n}]=\int_{K^m}[\Gamma_z]\wedge[\Gamma_{\mathfrak{a}_n}]\mathrm{d}\nu^{\otimes m}(z)\]
and only the geometrically transverse intersections are taken into account, i.e. for $\nu^{\otimes m}$-a.e. $z\in K^m$, the graphs $\Gamma_z$ and $\Gamma_{\mathfrak{a}_n}$ intersect transversely. In particular, this means there exists a sequence of parameters $\lambda_n\to \lambda_0$ and $z_n\in K^m$ such that the graph of $\mathfrak{a}_n$ and $\Gamma_{z_n}$ intersect transversely at $\lambda_n$.
Now, since repelling periodic points of $f^m_{\lambda_0}$ are dense in $K$, there exists $z_{n,j}\to z_n$ as $j\to\infty$, where $z_{j,n}\in K^m$ and $(f_{\lambda_0}^m,\ldots,f_{\lambda_0}^{m})$-periodic repelling. Since $z_{j,n}(\lambda):=(h_\lambda,\ldots,h_\lambda)(z_{j,n})$ remains in $(h_\lambda,\ldots,h_\lambda)(K^m)$ and remains periodic, it remains repelling for all $\lambda\in\B(\lambda_0,\epsilon)$. By persistence of transverse intersections, for $j$ large enough, there exists $\lambda_{j,n}$ where $\Gamma_{\mathfrak{a}_n}$ and $\Gamma_{z_{j,n}}$ intersect transversely and $\lambda_{j,n}\to \lambda_{n}$ as $j\to\infty$ and the proof is complete.
\end{proof}

To finish this section, we prove the Claim.

\begin{proof}[Proof of the Claim]
Since the compact set $K$ is contained in a ball, we can choose an affine chart $\C^{k}$ such that $K\subset \C^{k}$ and, up to reducing $\epsilon>0$, we can assume $K_\lambda=h_\lambda(K)\subset\C^{k}$ for all $\lambda\in\B(\lambda_0,\epsilon)$. Let $(x_1^1,\ldots,x_k^1,\ldots,x_1^m,\ldots,x_k^m)=(x^1,\ldots,x^m)$ be the coordinates of $(\C^k)^m$ and let $h_{\lambda,i}$ be the $i$-th coordinate of the function $h_\lambda$.

For all $1\leq i\leq k$ and $1\leq j\leq m$, we define a psh function $\Psi_i^j$ on  $\B(\lambda_0,\epsilon)\times (\C^k)^m$ by letting
\[\Psi_i^j(t,w):=\int_{K^m}\log|w_i^j-h_{t,i}(z^j)|\mathrm{d}\nu^{\otimes m}(z).\]
According to Lemma~\ref{lm:PLB} and Proposition~\ref{goodhyperbolicset}, we have \begin{center}
$\Psi_i^j\in L^\infty_{\mathrm{loc}}\left(\B(\lambda_0,\epsilon)\times(\C^{k})^m\right)$.
\end{center} Moreover, according to \cite[Theorem 3.1]{Dujardin2012}, we have
\[\hat{\nu}=\bigwedge_{i,j}dd^c\Psi_i^j=dd^{c}\left(\Psi_{1}^{1}\cdot\bigwedge_{i,j>1}dd^c\Psi_i^j\right).\]
Since the functions $\Psi_i^j$ are locally bounded, this ends the proof.
\end{proof}

\section{Local properties of bifurcation measures}\label{sec:similar}

\subsection{A renormalization procedure}\label{sec:renorm}
Pick $k,m\geq1$ and let $\B(0,\epsilon)$ be the open ball centered at $0$ of radius $\epsilon$ in $\C^{km}$ and let $(f,a_1,\ldots,a_m)$ be a dynamical $(m+1)$-tuple of degree $d$ of $\p^k$ parametrized by $\B(0,\epsilon)$.

Assume there are $m$ holomorphically moving $J$-repelling  periodic points $z_1,\ldots,z_m:\B(0,\epsilon)\to\p^k$ of respective period $q_j\geq1$ with $f^{n_j}_0(a_j(0))=z_j(0)$. We also assume that $(a_1,\ldots,a_m)$ are transversely prerepelling at $0$ and that $z_j(\lambda)$ is linearizable for all $\lambda\in\B(0,\epsilon)$ for all $j$.
Let $q:=\mathrm{lcm}(q_1,\ldots,q_m)$ and
\[L_\lambda:=(D_{z_1(\lambda)}(f_\lambda^{q}),\ldots,D_{z_m(\lambda)}(f_\lambda^{q})):\bigoplus_{j=1}^m T_{z_j(\lambda)}\p^k\longrightarrow\bigoplus_{j=1}^m  T_{z_j(\lambda)}\p^k\]
and denote by $\phi_\lambda=(\phi_{\lambda,1},\ldots,\phi_{\lambda,m}):(\C^k,0)\rightarrow((\p^k)^m,(z_1(\lambda),\ldots,z_m(\lambda)))$, where $\phi_{\lambda,j}$ is the linearizing coordinate of $f_\lambda^{q}$ at $z_j(\lambda)$.

\medskip

Denote by $\pi_j:(\p^k)^m\rightarrow\p^k$ the projection onto the $j$-th factor.
Up to reducing $\epsilon>0$, we can also assume there exists $r_j>0$ independent of $\lambda$ such that
\[f^{q}_{\lambda}\circ \phi_{\lambda,j}(z)=\phi_{\lambda,j}\circ D_{z_j(\lambda)}(f_\lambda^{q_j})(z), \ z\in \B(0,r_j),\]
 and $D_0\phi_{\lambda,j}:\C^k\rightarrow T_{z_j(\lambda)}\p^k$ is an invertible linear map.
 Up to reducing again $\epsilon$, we can also assume $f_\lambda^{n_j}(a_j(\lambda))$ always lies in the range of $\phi_{\lambda,j}$ for all $1\leq j\leq m$. Recall that we denoted $\mathfrak{a}_{\underline{n}}(\lambda)=(f_\lambda^{n_1}(a_1(\lambda)),\ldots,f^{n_m}_\lambda(a_m(\lambda)))$, where $\underline{n}=(n_1,\ldots,n_m)$ and for $\lambda\in\B(0,\epsilon)$, let
\begin{align*}
h(\lambda)  := & \phi_\lambda^{-1}\circ \mathfrak{a}_{\underline{n}}(\lambda)\\
 = & \left(\phi_{\lambda,1}^{-1}\left(f_\lambda^{n_1}(a_1(\lambda))\right),\ldots,\phi_{\lambda,m}^{-1}\left(f_\lambda^{n_m}(a_m(\lambda))\right)\right).
\end{align*}

\begin{lemma}
The map $h:\B(0,\epsilon)\rightarrow(\C^{km},0)$ is a local biholomorphism at $0$.
\end{lemma}

\begin{proof}
Recall that $f_0^{n_j}(a_j(0))=z_j(0)$. Write $h=(h_1,\ldots,h_m)$ with $h_j:\B(0,\epsilon)\rightarrow(\C^k,0)$ and let $b_j(\lambda):=f_\lambda^{n_j}(a_j(\lambda))$ for all $\lambda\in\B(0,\epsilon)$ so that $b_j(\lambda)=\phi_{\lambda,j}\circ h_j(\lambda)$ for all $\lambda\in\B(0,\epsilon)$. Since $\phi_{\lambda,j}(0)=z_j(\lambda)$, differentiating and evaluating at $\lambda=0$, we find 
\[D_0b_j=D_0z_j+D_0\phi_{0,j}\circ D_0h_j.\]
Now our tranversality assumption implies that
\[L:=\left((D_0b_1-D_0z_1),\ldots,(D_0b_m-D_0z_m)\right):\C^{km}\rightarrow \bigoplus_{j=1}^m T_{z_j(0)}\p^k\]
is invertible.
As a consequence, the linear map
\[D_0h=(D_0h_1,\ldots,D_0h_m)=-(D_0\phi_0)^{-1}\circ L:\C^{km}\rightarrow\C^{km}\]
is invertible, ending the proof.
\end{proof}

Up to reducing again $\epsilon$, we assume $h$ is a biholomorphism onto its image and let $r:=h^{-1}:h(\B(0,\epsilon))\rightarrow\B(0,\epsilon)$. Fix $\delta_1,\ldots,\delta_m>0$ so that $\B_{\C^k}(0,\delta_1)\times\cdots\times\B_{\C^k}(0,\delta_m) \subset h(\B(0,\epsilon))$. 

\medskip 
Finally, let $\Omega:=\B_{\C^k}(0,\delta_1)\times\cdots\times\B_{\C^k}(0,\delta_m)$ and, for any $n\geq1$, let
\[r_n(x):=r\circ L_0^{-n}(x), \ x\in \B_{\C^k}(0,\delta_1)\times\cdots\times\B_{\C^k}(0,\delta_m).\]
The main goal of this paragraph is the following.

\begin{proposition}\label{prop:transfer}
In the weak sense of measures on $\Omega$, we have
\[\prod_{j=1}^md^{k(n_j+nq)}\cdot(r_n)^*\left(T_{f,a_1}^k\wedge\cdots\wedge T_{f,a_m}^k\right)\underset{n\rightarrow\infty}{\longrightarrow} (\phi_0)^*\left(\bigwedge_{j=1}^m(\pi_j)^*\mu_{f_0}\right).\]
\end{proposition}
\par\noindent In plain words, we are proving that near the parameter $0$, the bifurcation measure is asymptotic to the maximal entropy measure (viewed through the linearizing coordinate). This is a measurable asymptotic similarity property.

To simplify notations, we let
\[\mathfrak{a}_{(n)}:=\mathfrak{a}_{\underline{n}+nq}, \ \text{with} \ \underline{n}+nq=(n_1+nq,\ldots,n_m+nq).\]

 \begin{lemma}\label{lm:cvrenorm}
The sequence $(\mathfrak{a}_{(n)}\circ r_n)_{n\geq1}$ converges uniformly to $\phi_0$ on $\Omega$.
 \end{lemma}

 \begin{proof}
Note first that
 \[\mathfrak{a}_{(0)}\circ r(x)=\left(f_{r(x)}^{n_1}(a_1(r(x))),\ldots,f_{r(x)}^{n_m}(a_m(r(x)))\right)=\phi_{r(x)}(x), \ x\in \Omega,\]
 by definition of $r$.

By definition, the sequence $(r_n)_{n\geq1}$ converges uniformly and exponentially fast to $0$ on $\Omega$, since we assumed $z_1(0),\ldots,z_m(0)$ are repelling periodic points and since $r(0)=0$. Moreover, $L_{r_n}\rightarrow L_0$ and $\phi_{r_n(x)}\rightarrow \phi_0$ exponentially fast. In particular,
\[\lim_{n\rightarrow\infty}L_{r_n(x)}^n\circ L_0^{-n}(x)=x\]
and the convergence is uniform on $\Omega$.
Fix $x\in \Omega$. Then
\begin{align*}
\mathfrak{a}_{(n)}\circ r_n(x) & =(f_{r_n(x)}^{qn},\ldots,f_{r_n(x)}^{qn})\left(\mathfrak{a}_{(0)}\circ r\circ  L^{-n}_0(x)\right)\\
& = (f_{r_n(x)}^{qn},\ldots,f_{r_n(x)}^{qn})\circ \phi_{r_n(x)}\left(L^{-n}_0(x)\right)\\
& = \phi_{r_n(x)}\left(L^n_{r_n(x)}\circ L^{-n}_0(x)\right)
\end{align*}
and the conclusion follows.
\end{proof}

\begin{proof}[Proof of Proposition \ref{prop:transfer}]
Recall that we can assume there exists a holomorphic family of non-degenerate homogeneous polynomial maps $F_\lambda:\C^{k+1}\rightarrow\C^{k+1}$ of degree $d$ such that, if $\pi:\C^{k+1}\setminus\{0\}\rightarrow\p^k$ is the canonical projection, then
\[\pi\circ F_\lambda=f_\lambda\circ \pi \ \text{on} \ \C^{k+1}\setminus\{0\}.\]
For $1\leq j\leq m$, let $\tilde{a}_j:\B(0,\epsilon)\rightarrow\C^{k+1}\setminus\{0\}$ be a lift of $a_j$, i.e. $a_j=\pi\circ \tilde{a}_j$. Recall that
\[\bigwedge_{j=1}^mT_{a_j}^k=\bigwedge_{j=1}^m\left(dd^cG_\lambda(\tilde{a}_j(\lambda))\right)^k.\]
For $1\leq j\leq m$, pick a open set $U_j\subset\p^k$ such that $\phi_{0,j}(B_{\C^k}(0,\delta_j))\subset U_j$ and such that there exists a section $\sigma_j:U_j\rightarrow\C^{k+1}\setminus\{0\}$ of $\pi$ on $U_j$. Let $U:=U_1\times\cdots\times U_k$ and $\sigma:=(\sigma_1,\ldots,\sigma_k):U\rightarrow(\C^{k+1}\setminus\{0\})^m$ so that $\phi_0(\Omega)\subset U$. According to Lemma~\ref{lm:cvrenorm}, there exists $n_0\geq1$ such that
\[\mathfrak{a}_{(n)}\circ r_n(\Omega)\subset U.\]
In other words, for any $x\in\Omega$, any $1\leq j\leq m$ and any $n\geq n_0$,
\[a^{n,j}(x):=f_{r_n(x)}^{n_j+nq}(a_j\circ r_n(x))\in U_j.\]
Moreover, for all $x\in\Omega$, we have
\begin{align*}
\pi\circ F_{r_n(x)}^{n_j+nq}(\tilde{a}_j\circ r_n(x)) & = f_{r_n(x)}^{n_j+nq}\circ\pi (\tilde{a}_j\circ r_n(x)) = f_{r_n(x)}^{n_j+nq}(a_j\circ r_n(x))\\
& = \pi\circ\sigma_j\left(a^{n,j}(x)\right).
\end{align*}
In particular, there exists a holomorphic function $u_{n,j}:\Omega\rightarrow\C^*$ such that 
\[F_{r_n(x)}^{n_j+nq}(\tilde{a}_j\circ r_n(x))=u_{n,j}(x) \cdot \sigma_j\circ a^{n,j}(x)\]
and
\begin{align*}
d^{nq+n_j}G_{r_n(x)}\left(\tilde{a}_j\circ r_n(x)\right) & =G_{r_n(x)}\left(F_{r_n(x)}^{n_j+nq}\left(\tilde{a}_j\circ r_n(x)\right)\right)\\
& = G_{r_n(x)}\left(\sigma_j\circ a^{n,j}(x)\right)+\log|u_{n,j}(x)|,
\end{align*}
for all $x\in\Omega$. Since $\log|u_{n,j}|$ is pluriharmonic on $\Omega$, the above gives
\begin{align*}
d^{nq+n_j}(r_n)^*T_{f,a_j}
& = dd^cG_{r_n(x)}\left(\sigma_j\circ a^{n,j}(x)\right),
\end{align*}
so that 
\begin{align*}
\mu_n & :=\prod_{j=1}^md^{k(n_j+nq)}\cdot(r_n)^*\left(T_{f,a_1}^k\wedge\cdots\wedge T_{f,a_m}^k\right)\\
& =\bigwedge_{j=1}^m\left(dd^cG_{r_n(x)}\left(\sigma_j\circ a^{n,j}(x)\right)\right)^k.
\end{align*}
Using again Lemma~\ref{lm:cvrenorm} gives
\[\mu_n\underset{n\rightarrow\infty}{\longrightarrow}\bigwedge_{j=1}^m\left(dd^cG_{0}\left(\sigma_j\circ \phi_{0,j}(x)\right)\right)^k=\bigwedge_{j=1}^m(\phi_{0,j})^*\mu_{f_0}.\]
This ends the proof since $\phi_{0,j}=\pi_j\circ \phi_0$ by definition of $\phi_0$.
\end{proof}

\subsection{Families with an absolutely continuous bifurcation measure}

Fix integers $k,m\geq1$ and $d\geq2$. The following is a consequence of the above renormalization process. 

\begin{proposition}\label{prop:rigid}
Let $(f,a_1,\ldots,a_m)$ be a dynamical $(m+1)$-tuple of degree $d$ of $\p^k$ parametrized by the unit Ball $\B$ of $\C^{km}$.  Assume that $a_1,\ldots,a_m$ are transversely $J$-prerepelling at $0$ to a $J$-repelling cycle of $f_0$ which moves holmorphically in $\B$ as a $J$-repelling cycle of $f_\lambda$ which is linearizable for all $\lambda\in\B$.
Assume in addition that the measure $\mu:=T_{f,a_1}^k\wedge\cdots\wedge T_{f,a_m}^k$ is absolutely continuous with respect to the Lebesgue measure on $\B$ and the Radon-Nikodym derivative $\frac{d\mu}{d\mathrm{Leb}}$ is continuous and $>0$ near $0$.

Then the measure $\mu_{f_0}$ is non-singular with respect to $\omega_{\p^k}^k$.
\end{proposition}

\begin{proof}
By assumption, we can write $\mu=h\cdot\mathrm{Leb}$ where $h:\mathbb{B}\to\R_+$ is a continuous function. Let $\Omega:=\B_{\C^k}(0,\delta_1)\times\cdots\times\B_{\C^k}(0,\delta_m)$, $r_n$ and $\phi_0$ be given as in Section~\ref{sec:renorm}. We can apply Proposition~\ref{prop:transfer}:
\begin{align*}
\prod_{j=1}^md^{k(n_j+nq)}h\circ r_n\cdot (r_n)^*{\mathrm{Leb}} = & \prod_{j=1}^m d^{k(n_j+nq)}\cdot(r_n)^*\mu\\
& \underset{n\rightarrow\infty}{\longrightarrow} (\phi_0)^*\left(\bigwedge_{j=1}^m(\pi_j)^*\mu_{f_0}\right).
\end{align*}
Since $\phi_0(0)=(z_1(0),\ldots,z_{m}(0))\in(\supp(\mu_{f_0}))^k$, the measure
\[(\phi_0)^*\left(\bigwedge_{j=1}^m(\pi_j)^*\mu_{f_0}\right)\]
has (finite) strictly positive mass in $\Omega$. In particular, the measure
\[d^{knqm}\cdot (r_n)^*\left(h\cdot \mathrm{Leb}\right)=d^{knqm}\cdot (h\circ r_n)\cdot \left(\Lambda_0^{-n}\right)^*\left(r^*{\mathrm{Leb}}\right)\]
converges to a non-zero finite mass positive measure on $\Omega$. As $r$ is a local holomorphic diffeomorphism, there exists a neighborhood of $0$ in $\B$ such that we have $r^*{\mathrm{Leb}}=v\cdot \mathrm{Leb}$ for some smooth function $v>0$. Whence
\[d^{knqm}\cdot (r_n)^*{\mathrm{Leb}}=d^{knqm}\cdot(h\circ r_n)\cdot (v\circ \Lambda_0^{-n}) \left(\Lambda_0^{-n}\right)^*\left(\mathrm{Leb}\right).\]
By the change of variable formula and Fubini,
\[\left(\Lambda_0^{-n}\right)^*(\mathrm{Leb})=\prod_{j=1}^m|\det D_{z_j(0)}(f_0^q)|^{-2nk}\cdot \mathrm{Leb}.\]
For all $n$, define a continuous function $\alpha_n:\mathbb{B}\to\R_+$ by letting
\[\alpha_n(x):=d^{knqm}\prod_{j=1}^m|\det D_{z_j(0)}(f_0^q)|^{-2nk}\cdot(h\circ r_n(x))\cdot (v\circ \Lambda_0^{-n}(x))\in\R_+.\]
By assumption, the measure $\alpha_n\cdot \mathrm{Leb}$ converges weakly  on $\Omega$ to a non-zero finite positive measure, whence $\alpha_n\to\alpha_\infty$, as $n\to\infty$, where $\alpha_\infty:\Omega\to\R_+$ is not identically zero.
As a consequence,
\[(\phi_0)^*\left(\bigwedge_{j=1}^m(\pi_j)^*\mu_{f_0}\right)=\alpha_\infty\cdot {\mathrm{Leb}}.\]
Using again Fubini, on $\Omega$, we find
\[(\phi_0)^*\left(\bigwedge_{j=1}^m(\pi_j)^*\mu_{f_0}\right)=\alpha_\infty\cdot \mathrm{Leb}_{\C^{k}}\boxtimes \cdots\boxtimes  \mathrm{Leb}_{\C^{k}}.\]
Finally, since as positive measures on $\phi_0(\Omega)$, we have
\[\bigwedge_{j=1}^m(\pi_j)^*\mu_{f_0}=\mu_{f_0}\boxtimes\cdots\boxtimes\mu_{f_0},\]
the measure $\mu_{f_0}$ is absolutely continuous with respect to $\mathrm{Leb}$ in an open set.
\end{proof}

We now want to deduce Theorem~\ref{tm:rigidlattes} from the above, using \cite{Zdunik} when $k=1$ and \cite{BertelootDupont} when $k>1$. In fact, they prove that
$f$ is a Latt\`es map if and only if the sum of its Lyapunov exponents $L(f)=\int_{\p^{k}}\log|\det(Df)|\mu_f$ is equal to $\frac{k}{2}\log d$. We use this characterization to prove Theorems \ref{tm:rigidlattes-dim1} and \ref{tm:rigidlattes}.

\begin{proof}[Proof of Theorems \ref{tm:rigidlattes-dim1} $\&$ \ref{tm:rigidlattes}]
Assume first that $\mu_{f,a}$ is absolutely continuous with respect to $\omega^{k}$ with a continuous Radon-Nikodym derivative on $M\setminus\mathcal{Z}$, where $\mathcal{Z}$ is an analytic subvariety.
Let $T$ be the set of parameters $\lambda\in M$ such that $a$ is transversely $J$-prerepelling at $\lambda$. The set $T$ is dense in $M$ by Theorem \ref{tm:density}. Recall that all repelling cycles of an endomorphism of $\p^1$ are linearizable and that we assumed all repelling $J$-cycle to be linearizable when $k>1$. We thus can apply Proposition \ref{prop:rigid} at all $\lambda\in T$ outside an analytic subvariety $\mathcal{Z}$ of $M$: this gives
that $\mu_{f_\lambda}$ is non-singular with respect to $\omega_{\p^{k}}^{k}$ for all $\lambda\in T'= T\setminus\mathcal{Z}$. 

We then apply Zdunik or Berteloot-Dupont Theorem (depending on wether $k=1$ or $k>1$): the measure $\mu_{f_\lambda}$ is non-singular with respect to $\omega_{\p^k}^k$ if and only if $f_\lambda$ is a Latt\`es example.
We have proven there exists a countable subset $T'$ which is dense in $M$ such that the map $f_\lambda$ is a Latt\`es map for all $\lambda\in T'$. In particular, $L(f_\lambda)=\frac{k}{2}\log d$ for all $\lambda\in T'$. As the function $\lambda\in M\mapsto L(f_\lambda)$ is continuous and $T'$ is dense in $M$, this implies $L(f_\lambda)=\frac{k}{2}\log d$ for all $\lambda\in M$, i.e. $f_\lambda$ is a Latt\`es map for all $\lambda\in M$.

\medskip

To conclude, we assume $f$ is a family of Latt\`es maps and the measure $\mu_{f,a}$ is not identically zero.
Let $\omega_{\p^k}$ be the Fubini-Study form on $\p^k$. For all $\lambda\in M$, there exists a function $u_\lambda:\p^k\to\R_+\cup\{+\infty\}$ such that 
\[\mu_{f_\lambda}=u_\lambda\cdot \omega_{\p^k}^{k}.\]
Let $u(\lambda,z):=u_{\lambda}(z)$ for all $(\lambda,z)\in M\times\p^k$. The above can be expressed as
\[\widehat{T}=u\cdot \hat{\omega}^{k},\]
where $\hat{\omega}=\pi_{\p^k}^*(\omega_{\p^k})$ and $\pi_{\p^k}:M\times\p^k\to\p^k$ is the canonical projection.

Up to taking a branched cover $M'\to M$,  there exists
\begin{itemize}
\item a family of abelian varieties $\pi:\mathcal{A}\to M$, i.e. a holomorphic map such that $A_\lambda:=\pi^{-1}\{\lambda\}$ is an abelian variety of dimension $k$,
\item a finite branched Galois cover $\Theta:\mathcal{A}\to M\times\p^k$ such that $\Theta|_{A_\lambda}:A_\lambda\to\{\lambda\}\times\p^k$ is a finite branched Galois cover and
\item an integer $n\geq2$ such that $\Theta \circ [n]=\hat{f}\circ \Theta$, where $\hat{f}(\lambda,z):=(\lambda,f_\lambda(z))$ and $[n]$ is the fiberwise multiplication by $n$.
\end{itemize}
 There exists a closed positive $(1,1)$-current $\Omega$ on $\mathcal{A}$ which is smooth and such that $\Omega^k|_{A_\lambda}$ is the Haar measure of $A_\lambda$. One can, for example construct $\Omega$ as
\[\Omega:=\lim_{N\to\infty}\frac{1}{N^2}[N]^*\alpha,\]
where $\alpha$ is any relatively ample continuous form.
It is known that, in this case, $\widehat{T}^k=\Theta_*(\Omega^k)$,
so that $\widehat{T}^k$ is smooth outside the set $\mathcal{V}(\Theta)$ of critical values of the map $\Theta$, which form an analytic subvariety of $M\times\p^k$, i.e. $u$ is smooth on $M\times\p^k\setminus\mathcal{V}(\Theta)$.

\medskip

Let $\sigma:M\to M\times\p^k$ be the map given by $\sigma(\lambda):=(\lambda,a(\lambda))$, for $\lambda\in M$. Take now a local chart $U\subset M$ and a local chart $V\subset\p^k$ so that $a(U)\subset V$ and $\omega_{\p^k}=dd^cv$ on $V$ where $v$ is smooth.
In $U\times V$, the above gives
\begin{align*}
(\pi_\Lambda)_*\left(\widehat{T}^{k}\wedge [\Gamma_a]\right) & =(\pi_\Lambda)_*\left(u\cdot (dd^c_{\lambda,z}v(z))^{k}\wedge [\Gamma_a]\right)\\
& =u(\lambda,a(\lambda)) \left(dd^c_{\lambda}(v\circ a(\lambda))\right)^{k}.
\end{align*}
Letting $h(\lambda):=u(\lambda,a(\lambda))$ and $w(\lambda):=v\circ a(\lambda)$, $w$ is smooth and
\[\mu_{f,a}=h\cdot (dd^cw)^{k} \ \text{ on} \ U.\]
Since $h$ is smooth on $U\setminus\sigma^{-1}(\mathcal{V}(\Theta))$, the conclusion follows.
\end{proof}

\begin{remark}\normalfont
Note that if $(f,a)$ is an algebraic dynamical pair, $M$ is a quasiprojective variety, $a$ is a rational function and the map $\pi$ from the proof is a morphism. Since $\mathcal{Z}:=\sigma^{-1}(\mathcal{V}(\Theta))$ is the pull-back of an algebraic subvariety by a section of $\pi$, the set $\mathcal{Z}$ is an algebraic subvariety of $M$.
\end{remark}

\section{Proof of the main result and concluding remarks}\label{sec:proof}
\subsection{$J$-stability and bifurcation of dynamical pairs on $\p^1$}

Recall that a family $f:\Lambda\times\p^1\longrightarrow \p^1$ of degree $d$ rational maps of $\p^1$ is \emph{$J$-stable} if all the repelling cycles can be followed holomorphically throughout the whole family $\Lambda$, i.e. if for all $n\geq1$, there exists $N\geq0$ and holomorphic maps $z_1,\ldots,z_N:\Lambda\rightarrow\p^1$ such that $\{z_1(\lambda),\ldots,z_N(\lambda)\}$ is exactly the set of all repelling cycles of $f_\lambda$ of exact period $n$ for all $\lambda\in\Lambda$.

Recall also that an endomorphism of $\p^1$ has a unique measure of maximal entropy $\mu_f$ and let $L(f):=\int_{\p^1}\log|f'|\mu_f$ be the Lyapunov exponents of $f$ with respect to $\mu_f$. By a classical result of Ma\~n\'e, Sad and Sullivan~\cite{MSS}, it is also locally equivalent to the existence of a unique holomorphic motion of the Julia set which is compatible with the dynamics, i.e. for $\lambda_0\in\Lambda$, there exists $h:\Lambda\times J_{f_{\lambda_0}}\longrightarrow \Lambda\times \p^1$ such that
\begin{itemize}
\item for any $\lambda\in \Lambda$, the map $h_{\lambda}:=h(\lambda,\cdot):J_{f_{\lambda_0}}\longrightarrow\p^1$ is a homeomorphism which conjugates $f_{\lambda_0}$ to $f_\lambda$, i.e. $h_{\lambda}\circ f_{\lambda_0}=f_{\lambda}\circ h_{\lambda}$ on $J_{f_{\lambda_0}}$,
\item for any $z\in J_{f_{\lambda_0}}$, the map $\lambda\mapsto h_\lambda(z)$ is holomorphic on $\Lambda$,
\item $h_{\lambda_0}$ is the identity on $J_{f_{\lambda_0}}$.
\end{itemize}

\begin{lemma}\label{lm:bifJstab}
Let $(f,a)$ be any dynamical pair of $\p^1$ of degree $d\geq2$ parametrized by the unit disk $\D$. If $f$ is $J$-stable and $\supp(\mu_{f,a})\neq\emptyset$, we have
\[\supp(\mu_{f,a})=\{\lambda\in \D\, ; \ a(\lambda)\in  J_{f_\lambda}\}.\]
\end{lemma}

\begin{proof}
Since $\mathrm{Bif}(f,a)=\supp(\mu_{f,a})\neq\emptyset$, the set $D$ of parameters $\lambda_0\in\D$ such that $a$ is transversely prerepelling at $\lambda_0$ is a non-empty countable dense subset of $\mathrm{Bif}(f,a)$. As $J$-repelling points of $f_{\lambda_0}$ are contained in $J_{f_{\lambda_0}}$, this gives $\mathrm{Bif}(f,a)\subset \{\lambda\in \D\, ; \ a(\lambda)\in J_{f_\lambda}\}$.

Pick now $\lambda_0\in \{\lambda\in \D\, ; \ a(\lambda)\in J_{f_\lambda}\}$ and assume $\lambda_0\notin \mathrm{Bif}(f,a)$. Set $a_n(\lambda):=f_\lambda^n(a(\lambda))$ for all $n\geq0$ and all $\lambda\in\D$. Let $h:\D\times J_{f_{0}}\to\p^1$ be the unique holomorphic motion of $J_{f_0}$ parametrized by $\D$ such that, if $h_\lambda:=h(\lambda,\cdot)$, then
\[h_\lambda\circ f_{0}=f_\lambda\circ h_\lambda \ \text{on} \ J_{f_{0}}.\]
Note that for all $z\in J_{f_{0}}$, the sequence $\{\lambda\mapsto h_\lambda(f_{0}^n(z))\}_n$ is a normal family on $\D$.

Beware that for all periodic point $z\in J_{f_{0}}$ of $f_{0}$, the function $z(\lambda):=h_\lambda(z)$ is a marking of $z$ as a periodic point of $f_\lambda$.
For all $s\in\D$, if we let $h_\lambda^s:=h_\lambda\circ h_{s}^{-1}$. The family $(h_\lambda^s)_\lambda$ is a holomorphic motion of $J_{f_s}$ which satisfies 
\[h_\lambda^s\circ f_{s}=f_\lambda\circ h_\lambda^s \ \text{on} \ J_{f_s},\]
for all $\lambda\in\D(s,1-|s|)$.
Since we assumed $\lambda_0\notin\mathrm{Bif}(f,a)$, there exists $\epsilon>0$ such that $\D(\lambda_0,\epsilon)\cap\mathrm{Bif}(f,a)=\emptyset$ and we can choose an affine chart of $\p^1$ such that $a_n(\lambda)$ and $h_\lambda^{\lambda_0}(a_n(\lambda_0))$ lie in this chart for all $n\geq1$ and all $\lambda\in\D(\lambda_0,\epsilon)$. 
For all $n$, set
\[s_n(\lambda):=a_n(\lambda)-h_\lambda^{\lambda_0}(a_n(\lambda_0)), \ \lambda\in\D(\lambda_0,\epsilon).\]
Assume first $s_m\equiv 0$ on $\D(\lambda_0,\epsilon)$ for some $m\geq0$. This implies $a_m(\lambda)=h_\lambda(a_m(0))$ for all $\lambda\in\D(\lambda_0,\epsilon)$. By the Isolated Zero Theorem, we thus have
\[a_m(\lambda)=h_\lambda(a_m(0)) \ \text{for all} \ \lambda\in\D.\]
As $h_\lambda\circ f_0=f_\lambda\circ h_\lambda$, this yields $a_n(\lambda)\equiv h_\lambda(a_n(0))$ for all $n\geq m$, and $(a_n)$ is a normal family on $\D$. This is a contradiction, since we assumed $\mathrm{Bif}(f,a)\neq\emptyset$. We thus may assume $s_m\not\equiv0$ on $\D(\lambda_0,\epsilon)$. In particular, up to reducing $\epsilon$, we may assume $s_m(\lambda)\neq0$ for all $\lambda\in\D(\lambda_0,\epsilon)\setminus\{\lambda_0\}$. Let $z_0:=a_m(\lambda_0)$. By Rouch\'e Theorem, there exists $\eta>0$ such that for any $z\in \D(z_0,\eta)\cap J_{f_{\lambda_0}}$, the function
\[s_{m,z}(\lambda):=a_m(\lambda)-h_\lambda^{\lambda_0}(z)\]
has finitely many isolated zeros in $\D(\lambda_0,\epsilon)$. As repelling periodic points are dense in $J_{f_{\lambda_0}}$, there exists $z_1\in\D(z_0,\eta)\cap J_{f_{\lambda_0}}$ which is $f_{\lambda_0}$-periodic and repelling. The implies there exists $\lambda_1\in\D(\lambda_0,\epsilon)$ such that $a$ is properly prerepelling at $\lambda_1$. Finally, Theorem~\ref{tm:densitypart1}  (or simply Montel Theorem in this case) gives $\lambda_1\in\mathrm{Bif}(f,a)$ ending the proof.
\end{proof}

Using Montel theorem, one can deduce Theorem~\ref{prop:totalimpliesJstable}.

\begin{proof}[Proof of Theorem~\ref{prop:totalimpliesJstable}]
Assume first $f$ is $J$-stable. Note first that $J_{f_\lambda}=\p^1$ is true for some $\lambda\in\Lambda$ if and only if it is true for all $\lambda\in\Lambda$ in this case.

Assume now that $J_{f_\lambda}\neq \p^1$ for all $\lambda\in \Lambda$. Denote by $F_{f_\lambda}:=\p^1\setminus J_{f_\lambda}$ the Fatou set of $f_\lambda$.
By \cite[Theorem 7.8]{McM-S}, there exists a countable union of proper analytic subvariety $S\subset \Lambda$ such that $\Lambda\setminus S$ is open and, for any topological disk $D\subset \Lambda\setminus S$ (centered at some $\lambda_0$), there exists a unique holomorphic motion $\phi:D\times\p^1\to\p^1$ which conjugates $f_{\lambda_0}$ to $f_\lambda$ on $\p^1$. In particular the set
$\{(\lambda,z)\in D\times\p^1\, : \ z\in F_{f_\lambda}\}$ is a non-empty open subset of $D\times \p^1$. As we assumed $\mathrm{Bif}(f,a)=\Lambda$, the sequence $\{\lambda\longmapsto f_\lambda^n(a(\lambda))\}_{n\geq1}$ is not a normal family on $D$. By Montel Theorem, there exists $n\geq1$, $1\leq i\leq p$ and $\lambda_1\in D$ such that $f_{\lambda_1}^n(a(\lambda_1))\in  F_{f_{\lambda_1}}$, hence $a(\lambda_1)\in F_{f_{\lambda_1}}$.
However, Lemma~\ref{lm:bifJstab} gives
\[D=\mathrm{Bif}(f,a)\cap D=\{\lambda\in D\, : \ a(\lambda)\in J_{f_\lambda}\},\]
whence $a(\lambda_1)\in J_{f_{\lambda_1}}$. This is a contradiction.
This implies $J_{f_\lambda}=\p^{1}$ for all $\lambda\in\Lambda$. Finally, by Lemma V.1 of \cite{MSS}, if $f$ is not trivial, this implies $f_\lambda$ has an invariant linefield on its Julia set for all $\lambda\in\Lambda$.

\smallskip

If $f$ is not $J$-stable, by Montel Theorem,  there exists a non-empty open set $U$ of $\Lambda$ such that $(f_\lambda)_{\lambda\in U}$ is $J$-stable with an attracting periodic $z_1,\ldots,z_p$ of period $p\geq3$, and we proceed as follows: pick a topological disk $D\subset U$. Then there exists holomorphic functions $z_1,\ldots,z_p:D\to\p^1$ which paramerize this attracting cycle. In particular, $z_i(\lambda)\neq z_j(\lambda)$ for all $i\neq j$ and all $\lambda\in D$. Since we assumed $\mathrm{Bif}(f,a)=\Lambda$, the sequence $\{\lambda\longmapsto f_\lambda^n(a(\lambda))\}_{n\geq1}$ is not a normal family on $D$. By Montel Theorem, there exists $n\geq1$, $1\leq i\leq p$ and $\lambda_0\in D$ such that
\[f_{\lambda_0}^n(a(\lambda_0))=z_i(\lambda_0).\]
By Lemma~\ref{lm:bifJstab}, since $\lambda_0\in\mathrm{Bif}(f,a)$ this implies $z_i(\lambda_0)\in J_{f_{\lambda_0}}$. This is a contradiction with the fact that $z_i$ is attracting.
\end{proof}

\subsection{Proof of Theorem \ref{tm:rigidLattes} and the isotrivial case}

\begin{proof}[Proof of Theorem \ref{tm:rigidLattes}]
Remark that points $1.$ and $2.$ are equivalent by Theorem~\ref{tm:density}.
We first prove 1. implies 4. Assume $\mathrm{Bif}(f,a)=\Lambda$. By Theorem~\ref{prop:totalimpliesJstable} the family $f$ is $J$-stable. As $\Lambda$ is a quasi-projective manifold, by \cite[Theorem~2.4]{McM-stable}, since $f$ is not isotrivial, $f$ is a family of Latt\`es maps. 

We now prove 4. implies 1. We thus assume that $f$ is a non-isotrivial family of Latt\`es and that $\mu_{f,a}$ is non-zero. 
Recall that, since $f$ is a family of Latt\`es maps, it is stable. By Lemma~\ref{lm:bifJstab}, the set $\mathrm{Bif}(f,a)$ coincides with $\{\lambda\in\Lambda\, : \ a(\lambda)=J_{f_\lambda}\}=\Lambda$ since $J_{f_\lambda}=\p^1$ for all $\lambda\in\Lambda$.

\medskip 

The equivalence between 3. and 4. follows from Theorem~\ref{tm:rigidlattes-dim1} and the equivalence between 1. and 4. 
\end{proof}

Recall that when $f$ is isotrivial, either $J_{f_\lambda}=\p^{1}$ for all $\lambda$, or $J_{f_\lambda}\neq\p^{1}$ for all $\lambda$. We conclude this section with the following easy proposition, which clarifies the case when $f$ is isotrivial.

\begin{proposition}\label{prop:isotrivial}
Let $f$ be an isotrivial algebraic family parametrized by an irreducible quasiprojective curve $\Lambda$ and let $a:\Lambda\to\p^1$ be such that the pair $(f,a)$ is unstable. 
The following are equivalent:
\begin{enumerate}
\item the Julia set of $f_\lambda$ is $J_{f_\lambda}=\p^1$ for all $\lambda\in\Lambda$,
\item the bifurcation locus of $(f,a)$ contains a non-empty open set,
\item the bifurcation locus of $(f,a)$ is $\mathrm{Bif}(f,a)=\Lambda$.
\end{enumerate}
\end{proposition}

\begin{remark}\normalfont
In fact, in the isotrivial case we can also prove the following are equivalent:
\begin{enumerate}
\item the family $f$ is an isotrivial family of Latt\`es maps,
\item the measure $\mu_{f,a}$ is absolutely continuous with respect to $\omega_\Lambda$.
\end{enumerate}
\end{remark}

\begin{proof}
If $\mathrm{Bif}(f,a)=\Lambda$, obviously, it contains a non-empty open subset of $\Lambda$.
Now, as $f$ is isotrivial, up to taking a finite branched cover of $\Lambda$ and up to conjugating $f$ by a family of M\"obius transformations, we can assume $f_\lambda=f_0$ for all $\lambda\in\Lambda$. In particular, it is a $J$-stable family and., applying Lemma~\ref{lm:bifJstab} in local charts, we find
\[\mathrm{Bif}(f,a)=\{\lambda\in \Lambda\, ; \ a(\lambda)\in  J_{f_0}\}=a^{-1}(J_{f_0}).\]
Since $\mathrm{Bif}(f,a)\neq\varnothing$, the holomorphic map $a$ has to be non-constant, whence it is open. In particular, if $\mathrm{Bif}(f,a)$ contains a non-empty open set, $J_{f_0}$ has to contain a nonempty open set and $J_{f_0}=\p^1$. Finally, if $J_{f_0}=\p^1$, then we clearly have $\mathrm{Bif}(f,a)=a^{-1}(\p^1)=\Lambda$.

Assume first $J_{f_\lambda}=\p^1$ for all $\lambda\in\Lambda$. 
When $\mu_{f,a}$ is absolutely continuous, the conclusion follows as in the proof of Theorem \ref{tm:rigidLattes}.
\end{proof}

 \subsection{Concluding remarks and questions}

\paragraph*{Dynamical pairs with a non-singular bifurcation measure}
First, when $k>1$, the statement of Theorem \ref{tm:rigidlattes} holds only if all repelling $J$-cycles are linearizable.

This results raises several questions:

\begin{question}
Can we generalize Theorem \ref{tm:rigidlattes} to the cases when
\begin{enumerate}
\item There exists $J$-repelling cycles that are non-linearizable?
\item $T_{f,a}^k$ is just \emph{non-singular} with respect to a smooth volume form?
\end{enumerate}
\end{question}

The first question is very likely to have a positive answer, using Poincar\'e-Dulac normal forms instead of linear normal forms. However, it looks quite difficult to prove rigorously.

\medskip

In fact, Zdunik \cite{Zdunik} completely classifies rational maps with a maximal entropy measure which is not singular with respect to a Hausdorff measure $\mathcal{H}^{\alpha}$: either $\alpha=1$ and the rational map is conjugated to a monomial map $z^{\pm d}$ or to a Chebichev polynomial $T_d$, i.e. the only polynomial satisfying 
$T_d\left(z+\frac{1}{z}\right)=z^d+\frac{1}{z^d}$ for all $z\in\C$, or $\alpha=2$ and the rational map is a Latt\`es map.

We expect the following complete parametric counterpart to \cite{Zdunik} to be true:

\begin{question}\label{tm:classification}
Let $(f,a)$ be any holomorphic dynamical pair of $\p^1$ of degree $d\geq2$ parametrized by the unit disk $\D$ of $\C$. Assume that $(f,a)$ is unstable.
Assume also there exists $\alpha>0$ and a function $h:\D\to\R_+$ such that $\mu_{f,a}=h\cdot \mathcal{H}^{\alpha}$ on $\D$.
Can we prove that 
\begin{itemize}
\item either $\alpha=2$ and $f$ is a family of Latt\`es maps,
\item or $\alpha=1$, $f$ is isotrivial and all $f_\lambda$'s are conjugated to $z^{\pm d}$ or a Chebichev polynomial?
\end{itemize}
\end{question}

As in the case of families of Latt\`es maps, we can expect the proof to generalize to the case when $k>1$. This raises the following question.

\begin{question}
Classify endomorphisms of $\mathbb{P}^{k}$ which maximal entropy measure is not singular with respect to some Hausdorff measure $\mathcal{H}^\alpha$ on $\mathbb{P}^{k}$ (and 
possible values of $\alpha$).
\end{question}

As seen above, the case $\alpha=2k$ has been treated by Berteloot and Loeb \cite{BL} and Berteloot and Dupont \cite{BertelootDupont}. Of course, there are also easy examples where $\alpha=k$: take $f:\mathbb{P}^1\to\mathbb{P}^1$ which maximal entropy measure has dimension $1$, then 
the endomorphism
$F:\p^{k}\longrightarrow\p^{k}$ making the following diagram commute
\begin{center}
$\xymatrix {\relax
(\p^1)^k \ar[r]^{(f,\ldots,f)} \ar[d]_{\eta_k} & (\p^1)^k\ar[d]^{\eta_k} \\
\p^{k} \ar[r]
_{F} & \p^{k}}$
\end{center}
where $\eta_k$ is the quotient map of the action by permutation of coordinates of the symmetric group $\mathfrak{S}_k$, satisfies $\dim(\mu_F)=k$ (see \cite{GHK-Symm} for a study of symmetric products).

\paragraph*{$J$-stability and dynamical pairs, when $k\geq2$}
We say that a family $f:\Lambda\times\p^k\longrightarrow\p^k$ of degree $d\geq2$ endomorphisms of $\p^k$ is \emph{weakly $J$-stable} if all the $J$-repelling cycles can be followed holomorphically throughout the whole family $\Lambda$, i.e. if for all $n\geq1$, there exists $N\geq0$ and holomorphic maps $z_1,\ldots,z_N:\Lambda\rightarrow\p^k$ such that $\{z_1(\lambda),\ldots,z_N(\lambda)\}$ is exactly the set of all repelling $J$-cycles of $f_\lambda$ of exact period $n$ for all $\lambda\in\Lambda$.

\medskip

For any endomorphism $f$ of $\p^k$, let $L(f):=\int_{\p^k}\log|\det Df|\mu_f$ be the sum of the Lyapunov exponents of $f$ with respect to its Green measure $\mu_f$. By a beautiful result of Berteloot, Bianchi and Dupont~\cite{BBD}, $f$ is $J$-stable if and only if $\lambda\longmapsto L(f_\lambda)$ is pluriharmonic on $\Lambda$.

A natural question is then the following:

\begin{question}
Given any dynamical pair $(f,a)$ of degree $d$ of the projective space $\mathbb{P}^{k}$ parametrized by the unit ball $\mathbb{B}\subset\C^k$ such that $f$ is a weakly $J$-stable family, do we still have
\[\mathrm{Supp}(T_{f,a}^{k})=\{\lambda\in\mathbb{B}\, ; \ a(\lambda)\in J_{f_\lambda}\} \ ?\]
\end{question}

Note that this holds for $k=1$ by Lemma~\ref{lm:bifJstab}. One of the difficulties, when $k>1$, is that the weak $J$-stability is equivalent to the existence of a \emph{branched} holomorphic motion.

\bibliographystyle{short}
\bibliography{biblio}

\begin{thebibliography}{DMWY}

\bibitem[AGMV]{AGMV}
Matthieu Astorg, Thomas Gauthier, Nicolae Mihalache, and Gabriel Vigny.
\newblock Collet, {E}ckmann and the bifurcation measure.
\newblock {\em Invent. Math.}, 217(3):749--797, 2019.

\bibitem[BBD]{BBD}
Fran\c{c}ois Berteloot, Fabrizio Bianchi, and Christophe Dupont.
\newblock Dynamical stability and {L}yapunov exponents for holomorphic
  endomorphisms of {$\mathbb{P}^k$}.
\newblock {\em Ann. Sci. \'Ec. Norm. Sup\'er. (4)}, 51(1):215--262, 2018.

\bibitem[BD1]{BD-unlikely}
Matthew Baker and Laura DeMarco.
\newblock Preperiodic points and unlikely intersections.
\newblock {\em Duke Math. J.}, 159(1):1--29, 2011.

\bibitem[BD2]{BD}
Matthew Baker and Laura DeMarco.
\newblock Special curves and postcritically finite polynomials.
\newblock {\em Forum Math. Pi}, 1:e3, 35, 2013.

\bibitem[BD3]{BertelootDupont}
F.~Berteloot and C.~Dupont.
\newblock Une caract\'erisation des endomorphismes de {L}att\`es par leur
  mesure de {G}reen.
\newblock {\em Comment. Math. Helv.}, 80(2):433--454, 2005.

\bibitem[BD4]{briendduval}
Jean-Yves Briend and Julien Duval.
\newblock Exposants de {L}iapounoff et distribution des points p\'eriodiques
  d'un endomorphisme de {$\mathbf{C}{\rm P}^k$}.
\newblock {\em Acta Math.}, 182(2):143--157, 1999.

\bibitem[BE]{buffepstein}
Xavier Buff and Adam~L. Epstein.
\newblock Bifurcation measure and postcritically finite rational maps.
\newblock In {\em Complex dynamics : families and friends / edited by Dierk
  Schleicher}, pages 491--512. A K Peters, Ltd., Wellesley, Massachussets,
  2009.

\bibitem[BL]{BL}
Fran{\c{c}}ois Berteloot and Jean-Jacques Loeb.
\newblock Une caract\'erisation g\'eom\'etrique des exemples de {L}att\`es de
  {${\mathbb{ P}}^k$}.
\newblock {\em Bull. Soc. Math. France}, 129(2):175--188, 2001.

\bibitem[De1]{DeMarco2}
Laura DeMarco.
\newblock Dynamics of rational maps: {L}yapunov exponents, bifurcations, and
  capacity.
\newblock {\em Math. Ann.}, 326(1):43--73, 2003.

\bibitem[De2]{DeMarco-heights}
Laura DeMarco.
\newblock Bifurcations, intersections, and heights.
\newblock {\em Algebra Number Theory}, 10(5):1031--1056, 2016.

\bibitem[Duj]{Dujardin2012}
Romain Dujardin.
\newblock The supports of higher bifurcation currents.
\newblock {\em Ann. Fac. Sci. Toulouse Math. (6)}, 22(3):445--464, 2013.

\bibitem[Dup]{DupontThese}
Christophe Dupont.
\newblock {\em {Propri{\'e}t{\'e}s extr{\'e}males et caract{\'e}ristiques des
  exemples de Latt{\`e}s}}.
\newblock {PhD Thesis under the direction of Fran\c{c}ois Berteloot},
  {Universit{\'e} Paul Sabatier - Toulouse III}, November 2002.

\bibitem[DF]{favredujardin}
Romain Dujardin and Charles Favre.
\newblock Distribution of rational maps with a preperiodic critical point.
\newblock {\em Amer. J. Math.}, 130(4):979--1032, 2008.

\bibitem[DM]{DeMarco-Mavraki}
Laura DeMarco and Niki~Myrto Mavraki.
\newblock Variation of canonical height and equidistribution.
\newblock {\em Amer. J. Math.}, 142(2):443--473, 2020.

\bibitem[dMvS]{demelovanstrien}
Welington de~Melo and Sebastian van Strien.
\newblock {\em One-dimensional dynamics}, volume~25 of {\em Ergebnisse der
  Mathematik und ihrer Grenzgebiete (3) [Results in Mathematics and Related
  Areas (3)]}.
\newblock Springer-Verlag, Berlin, 1993.

\bibitem[DMWY]{DWY}
Laura De~Marco, Xiaoguang Wang, and Hexi Ye.
\newblock Bifurcation measures and quadratic rational maps.
\newblock {\em Proc. Lond. Math. Soc. (3)}, 111(1):149--180, 2015.

\bibitem[DS]{DS-PLM}
Tien-Cuong Dinh and Nessim Sibony.
\newblock Dynamique des applications d'allure polynomiale.
\newblock {\em J. Math. Pures Appl. (9)}, 82(4):367--423, 2003.

\bibitem[DTGV]{dTGV}
Henry De~Th\'{e}lin, Thomas Gauthier, and Gabriel Vigny.
\newblock The bifurcation measure has maximal entropy.
\newblock {\em Israel J. Math.}, 235(1):213--243, 2020.

\bibitem[FG1]{favregauthier}
Charles Favre and Thomas Gauthier.
\newblock Distribution of postcritically finite polynomials.
\newblock {\em Israel Journal of Mathematics}, 209(1):235--292, 2015.

\bibitem[FG2]{specialcubic}
Charles Favre and Thomas Gauthier.
\newblock Classification of special curves in the space of cubic polynomials.
\newblock {\em Int. Math. Res. Not. IMRN}, (2):362--411, 2018.

\bibitem[FG3]{Conti}
Charles Favre and Thomas Gauthier.
\newblock Continuity of the {G}reen function in meromorphic families of
  polynomials.
\newblock {\em Algebra Number Theory}, 12(6):1471--1487, 2018.

\bibitem[G]{Article1}
Thomas Gauthier.
\newblock Strong bifurcation loci of full {H}ausdorff dimension.
\newblock {\em Ann. Sci. \'Ec. Norm. Sup\'er. (4)}, 45(6):947--984, 2012.

\bibitem[GHK]{GHK-Symm}
T.~{Gauthier}, B.~{Hutz}, and S.~{Kaschner}.
\newblock {Symmetrization of Rational Maps: Arithmetic Properties and Families
  of Latt\`es Maps of $\mathbb P^k$}.
\newblock {\em ArXiv e-prints}, March 2016.

\bibitem[M]{McM-stable}
Curt McMullen.
\newblock Families of rational maps and iterative root-finding algorithms.
\newblock {\em Ann. of Math. (2)}, 125(3):467--493, 1987.

\bibitem[MS]{McM-S}
Curtis~T. McMullen and Dennis~P. Sullivan.
\newblock Quasiconformal homeomorphisms and dynamics. {III}. {T}he
  {T}eichm\"{u}ller space of a holomorphic dynamical system.
\newblock {\em Adv. Math.}, 135(2):351--395, 1998.

\bibitem[MSS]{MSS}
R.~Ma{\~n}{\'e}, P.~Sad, and D.~Sullivan.
\newblock On the dynamics of rational maps.
\newblock {\em Ann. Sci. \'Ecole Norm. Sup. (4)}, 16(2):193--217, 1983.

\bibitem[R]{Rees}
Mary Rees.
\newblock Positive measure sets of ergodic rational maps.
\newblock {\em Ann. Sci. \'Ecole Norm. Sup. (4)}, 19(3):383--407, 1986.

\bibitem[T]{similarity}
Lei Tan.
\newblock Similarity between the {M}andelbrot set and {J}ulia sets.
\newblock {\em Comm. Math. Phys.}, 134(3):587--617, 1990.

\bibitem[Z]{Zdunik}
Anna Zdunik.
\newblock Parabolic orbifolds and the dimension of the maximal measure for
  rational maps.
\newblock {\em Invent. Math.}, 99(3):627--649, 1990.

\end{thebibliography}
\end{document}